\documentclass[a4paper]{article}
\usepackage{amsmath,amssymb,amsfonts,amsthm} % 引入AMS的数学、符号、字体和定理宏包
\usepackage{mathrsfs} % 引入花体字母宏包
\usepackage{enumerate} % 用于创建自定义的列表环境
\usepackage{cases} % 用于创建多行方程中的分段函数
\usepackage{float} % 提供更多控制浮动对象的位置的选项
\usepackage{graphicx} % 引入图形宏包 Include this line if your document contains figures,
\usepackage{subfig}
\allowdisplaybreaks[4] % 允许多行公式在页面之间断开
\usepackage{wrapfig}  % 允许图形和文字环绕
\usepackage{url} % 支持URL格式
\usepackage[usenames]{color} % 引入颜色宏包，并使用预定义的颜色名称
\usepackage{setspace} % 设置行间距
\usepackage{booktabs}
\def\b1{\mbox{\boldmath $1$}} % 定义加粗的数字1

\makeatletter
\newcommand{\Biggg}{\bBigg@{3.5}} % 定义一个新的大尺寸符号
\makeatother

\theoremstyle{plain} % 设置定理样式为plain
\newtheorem{theorem}{\bf Theorem}[section] % 定义新的定理环境，编号基于章节
\newtheorem{lemma}{\bf Lemma}[section] % 定义新的引理环境，编号与定理共享
\newtheorem{proposition}{\bf Proposition}[section] % 定义新的命题环境，编号与定理共享
\newtheorem{definition}{\bf Definition}[section] % 定义新的定义环境，编号与定理共享
\newtheorem{remark}{\bf Remark}[section]  % 定义新的备注环境，编号与定理共享
\makeatother
\allowdisplaybreaks %  允许多行公式在页面之间断开.

\usepackage[numbers,sort&compress]{natbib} % 引入natbib宏包，使参考文献连续编号的引用以压缩方式出现
\usepackage[bookmarksnumbered, bookmarksopen,colorlinks,citecolor=blue,linkcolor=blue]{hyperref} % 引入超链接宏包，并设置超链接属性

\begin{document}
\date{}

\title{Closed-loop solvability of infinite-horizon stochastic linear-quadratic problem for Markov regime-switching jump-diffusion system\thanks{Fan Wu is supported by the Natural Science Foundation of Anhui Province (Grant No. 2508085QA026) and the Provincial Natural Science Research Project of Anhui Colleges (Grant No. 2025AHGXZK30544).Jie Xiong is supported by National Key R\&D Program of China (Grant No. 2022YFA1006102), and the National Natural Science Foundation of China (Grant No. 12471418).}}
\author{Kai Ding \thanks{Department of Mathematics, Southern University of Science and Technology, Shenzhen 518055, China; E-mail: dkaiye@163.com.}\qquad
Fan Wu \thanks{Corresponding Author: School of Big Data and Statistics, Anhui University, Hefei 230601, China; E-mail: wfyy121107@163.com.}\qquad
Jie Xiong \thanks{Department of Mathematics and SUSTech International Center for Mathematics, Southern University of Science and Technology, Shenzhen, Guangdong, 518055, China; E-mail: xiongj@sustech.edu.cn.}\qquad 
Xinyue Zhang \thanks{School of Big Data and Statistics, Anhui University, Hefei 230601, China.}} 
\maketitle

\noindent{\bf Abstract:} This paper investigates a class of stochastic linear-quadratic (SLQ) control problems over an infinite horizon for Markov regime-switching jump-diffusion systems. Unlike classical diffusion models modulated by a Markov chain, we assume that the state process undergoes abrupt jumps that are synchronous with the regime switches of the Markov chain. In contrast to conventional Poisson jump-diffusion models, the jumps in the state process are entirely induced by the state transitions of the Markov chain, which can be interpreted as losses or gains of state process incurred during regime changes. Under this formulation, we thoroughly discuss the closed-loop solvability of the SLQ control problem and provide a feedback representation of the optimal control via the stabilizing solution of a system of coupled algebraic Riccati equations (CAREs). Finally, we further apply our results to a lifetime wealth tracking problem and derive the corresponding optimal investment strategy.

%A closed-form equilibrium reinsurance-investment strategy is obtained in our paper.
\medskip



\section{Introduction}
Let $(\Omega,\mathcal{F},\mathbb{F},\mathbb{P} )$ be a complete filtered probability space on which a standard one-dimensional Brownian motion $W=\{W(t)\}_{t\geq 0}$ and a continuous time irreducible Markov chain $\alpha=\{\alpha_t\}_{t\geq 0}$  are defined with $\mathbb{F}=\{\mathcal{F}_{t}\}_{t\geq 0}$ being its natural filtration augmented by all $\mathbb{P}$-null sets in $\mathcal{F}$. Throughout this paper, we denote the state space of the Markov chain $\alpha(\cdot)$ as $\mathcal{S}:=\left\{1,2,...,L\right\}$, where $L$ is a finite natural number. The generator of Markov chain $\alpha$ is given by $\mathbf{\Pi}=\left\{\pi_{ij}\mid i,j\in\mathcal{S}\right\}$,
where $\pi_{ij}\geq 0$ for $i\neq j $ and $\pi_{ii}=-\sum_{j\neq i}\pi_{ij}$.
Let $\mathbb{I}_{A}$ be the indicator function, and $N_{j}(t)$ be the number of jumps to state $j$ up to time $t$. Then
$$\widetilde{N}_{j}(t)\triangleq N_{j}(t)-\int_{0}^{t}\lambda_{j}(s)ds\quad \text{ with }\quad \lambda_{j}(s)=\sum_{i=1,i\neq j}^{L}\left[\pi_{ij}\mathbb{I}_{\{\alpha_{s-}=i\}}\right],$$
is an $\left(\mathbb{F},\mathbb{P}\right)$-martingale for each $j\in\mathcal{S}$. We define $$\Theta(\alpha_{t})\triangleq\sum_{i=1}^{L}\Theta(i)\mathbb{I}_{(\alpha_{t}=i)},$$ 
for any given $L$-dimension vector $\mathbf{\Theta}=[\Theta(1),\Theta(2),...,\Theta(L)]$. Let $\mathbb{R}^{n\times m}$ denote the Euclidean space of all $ n \times m$  matrices and set $\mathbb{R}^{n}:=\mathbb{R}^{n\times 1}$ for simplicity. In addition, the set of all $ n\times n$ symmetric matrices is denoted by $\mathbb{S}^n$. Specially, the sets of all $ n\times n$ semi-positive definite matrices and positive definite matrices are denoted by $\overline{\mathbb{S}_{+}^n}$ and $\mathbb{S}_{+}^n$, respectively. 
 
Now, consider the following  controlled linear stochastic differential equation (SDE) with Markovian jumps over the infinite time horizon $[0,\infty)$:
 \begin{equation}\label{intro-state}
   \left\{
   \begin{aligned}
   dX(t)&=\left[A\left(\alpha_{t}\right)X(t)+B\left(\alpha_{t}\right)u(t)\right]dt
   +\left[C\left(\alpha_{t}\right)X(t)+D\left(\alpha_{t}\right)u(t)\right]dW(t)\\
   &\quad+\sum_{j=1}^{L}\left[E_{j}\left(\alpha_{t-}\right)X(t-)+F_{j}\left(\alpha_{t-}\right)u(t)\right]d\widetilde{N}_{j}(t),\\
   X(0)&=x,\quad \alpha(0)=i,\qquad t\geq 0,
   \end{aligned}
   \right.
 \end{equation}
where $A(i), \, C(i), \,  E_{j}(i)\in  \mathbb{R}^{n \times n}$, $B(i), \, D(i), \, F_{j}(i)\in \mathbb{R}^{n \times m}$ for any given $i,j\in\mathcal{S}$.  In the above, $X(\cdot)$ is called the state process, which valued in $\mathbb{R}^{n}$, and $u(\cdot)$ is called the control process, which valued in $\mathbb{R}^{m}$. The goal  of this paper
is to find an optimal control to minimize the following cost functional:
\begin{equation}\label{intro-cost}
\begin{aligned}
    J\left(x,i;u(\cdot)\right)
    & \triangleq \mathbb{E}\int_{0}^{\infty}
    \left<
    \left(
    \begin{matrix}
    Q(\alpha_{t})&S(\alpha_{t})^{\top}\\
    S(\alpha_{t})&R(\alpha_{t})
    \end{matrix}
    \right)
    \left(
    \begin{matrix}
    X(t)\\
    u(t)
    \end{matrix}
    \right),
    \left(
    \begin{matrix}
    X(t)\\
    u(t)
    \end{matrix}
    \right)
    \right>dt.
  \end{aligned}
\end{equation}
Here, for $i\in\mathcal{S}$, $Q(i)\in\mathbb{S}^{n}$, $R(i)\in\mathbb{S}^{m}$ and $S(i)\in\mathbb{R}^{m\times n}$. It is worth mentioning that the performance coefficients $Q(\cdot) $ and $R(\cdot) $ in the above are not necessarily (semi) positive definite matrices. Hence, we are about to solve an indefinite SLQ control problem over infinite horizon.

%Contrast to the finite time horizon SLQ problem, an important difference in the infinite time horizon SLQ problem is that even if the control process is square integrable, it is difficult to guarantee that the state equation is square integrable, and hence, it cannot guarantee that the performance functional is well-defined. 
By a standard argument using contraction mapping theorem, we can show that for any initial state $(x,i)\in\mathbb{R}^{n}\times\mathcal{S}$ and control $u(\cdot)\in L_{\mathcal{P}}^{2}(\mathbb{R}^{m})$, the state equation \eqref{intro-state} admits a unique solution $X(\cdot)=X(\cdot;x,i,u)\in L_{\mathbb{F}}^{2,loc}(\mathbb{R}^{n})$. To ensure the cost  functional is well-defined, we introduce the following set:
$$\mathcal{U}_{ad}(x,i) \triangleq \left\{u(\cdot)\in L_{\mathcal{P}}^{2}(\mathbb{R}^{m}) \mid X(\cdot;x,i,u)\in L_{\mathbb{F}}^{2}(\mathbb{R}^{n})\right\}, \quad (x,i)\in \mathbb{R}^{n}\times\mathcal{S}. $$
An element $u(\cdot)\in \mathcal{U}_{ad}(x,i)$ is called an admissible control for the initial state $(x,i)$.  Based on the above notations, we introduce the following problem.

\noindent\textbf{Problem (M-SLQ)$_{\infty}$.} For any given $(x,i)\in \mathbb{R}^{n}\times \mathcal{S}$, find a $u^{*}(\cdot)\in\mathcal{U}_{ad}(x,i)$ such that
\begin{equation}\label{value}
     J\left(x,i;u^{*}(\cdot)\right)=\inf_{u(\cdot)\in\mathcal{U}_{ad}(x,i)}J\left(x,i;u(\cdot)\right)\triangleq V(x,i).
\end{equation}
In the above, the element $u^{*}(\cdot)\equiv u(\cdot;x,i)\in \mathcal{U}_{ad}(x,i)$ is called an open-loop optimal control for Problem (M-SLQ)$_{\infty}$ at initial value $(x,i)$, and the function $V(\cdot,\cdot)$ is called the value function of Problem (M-SLQ)$_{\infty}$.

The study of SLQ control problems can be traced back to the pioneering work of Bismut \cite{Bismut.J.M.1976_LQR}. Since then, SLQ control has attracted considerable attention due to its wide applications in engineering, finance, insurance, economics, and related fields. In the finite horizon setting, Tang \cite{Tang.S.J.2003_LQR} investigated SLQ control problem with random coefficients and established the existence and uniqueness of solutions to the associated backward stochastic Riccati equations under standard assumption. Yong \cite{yong_linear-quadratic_2013} studied mean-field SLQ control problem governed by diffusion system and derived a feedback representation of the optimal control by solving the corresponding differential Riccati equations. 
Zhang et al. \cite{ZhangDongMeng2020BSREJ} investigated an SLQ control problem with Poisson jumps and derived a feedback representation of the optimal control.
For infinite horizon SLQ control problems, early contributions date back to the works of Ait Rami et al. \cite{Rami-Zhou-Moore-2000-ID-LQ-IF} and Ait Rami and Zhou \cite{Rami-Zhou-2000-LMI-RE-IDLQIF}. Under the assumption that the state process of the control system is mean-square stable, they successfully solved the indefinite SLQ control problem for diffusion system by means of linear matrix inequalities and semidefinite programming techniques. Along this line, Li et al. \cite{Li-Zhou-Rami-2003-ID-MLQ-IF} extended the results of \cite{Rami-Zhou-Moore-2000-ID-LQ-IF,Rami-Zhou-2000-LMI-RE-IDLQIF} to SLQ control problem for diffusion systems modulated by a Markov chain. Huang et al. \cite{Jianhui-Huang-2015} obtained the optimal feedback control for mean-field SLQ problem under the assumption that the state process is $L^{2}$-stabilizable. Sun and Yong \cite{Sun-Yong-2018-ISLQI} further investigated the open-loop and closed-loop solvability of infinite horizon SLQ control problems, and proved that both types of solvability are equivalent to the existence of a static stabilizing solution to a constrained algebraic Riccati equation. In our previous work \cite{WuLiZhang2025}, we have developed the corresponding results in \cite{Sun-Yong-2018-ISLQI} to the SLQ control problem for Markov-chain-modulated diffusion system in infinite horizon. For further studies on the infinite-horizon SLQ control problems, interested readers may refer to references \cite{Chen-Xi-2004-stochastic,PuZhang2021ConstrainedSLQ, Yao-Zhang-ZHou-2004,bai-2020-linear,Li-Shi-Yong-2021-ID-MFLQ-IF,Wei-Qingmeng-2021-infinite,li-2022-reinforcement}, among others.
However, it is worth noting that the introduction of Markovian jumps into the state system will induce a fundamental change in the control problem.

In the classical Markov-modulated diffusion setting, the state equation is usually of the form
\begin{equation}
\left\{
\begin{aligned}
dX(t) &= \big[A(\alpha_{t})X(t)+B(\alpha_{t})u(t)\big]dt
       +\big[C(\alpha_{t})X(t)+D(\alpha_{t})u(t)\big]dW(t), \\
X(0)&=x,\qquad \alpha(0)=i.
\end{aligned}
\right.
\label{eq:classical_mjls}
\end{equation}
In the above, %the Markov chain modulates the system coefficients, %such as the drift matrix, diffusion matrix, control matrix, and weighting matrices in the cost functional. 
%and the main feature of \eqref{eq:classical_mjls} is that the system evolves continuously between and across regime switches. A
although the coefficients change when the Markov chain jumps from one state to another, the continuous state process $X(\cdot)$ itself does not experience an instantaneous jump at the switching time. Therefore, the classical model is suitable for situations in which the operating environment changes randomly, but the state variable adjusts continuously. In contrast, the system \eqref{intro-state} contains an additional jump term driven by the compensated counting processes associated with the regime switches. Thus, a transition of the Markov chain is no longer merely a change of coefficients; it may also generate an instantaneous jump in the continuous state variable.
This distinction is substantial. In the classical Markov-modulated diffusion model, a transition from regime $i$ to regime $j$ changes the future dynamics from
\[
(A(i),B(i),C(i),D(i))
\quad\hbox{to}\quad
(A(j),B(j),C(j),D(j)),
\]
but the state remains continuous. In the regime-switching jump-diffusion model, the same transition may additionally produce
\[
\Delta X(t)
=
E_{j}(i)X(t-)+F_{j}(i)u(t),
\]
when $\alpha(t-)=i,\ \alpha(t)=j$.
Consequently, the proposed model captures both the persistent effect of regime changes on future dynamics and the instantaneous shock caused by the switching event itself. Here, the jump magnitude $\Delta X(t)$ can be regarded as a loss or gain incurred by the state process accompanying the regime switching. Mathematically, introducing Markovian jumps into the state equation not only modifies the infinite-horizon stability of the state system but also affects the admissible control set and the convexity of the cost functional. In particular, for the state system \eqref{intro-state}, we will derive a more intricate system of CAREs, which are expected to incorporate additional nonlinear terms involving the jump matrices $E_{j}(i)$ and $F_{j}(i)$.

% the Markovian jumps lead to a fundamentally different structure in the dynamic programming equation, and the associated Riccati system. Specifically, in the classical model, the Markov chain contributes the coupling terms of value function in the dynamic programming equation as
% \[
% \sum_{j=1}^{L}\pi_{ij}\big[V(x,j)-V(x,i)\big].
% \]
% However, in the regime-switching jump-diffusion case, the corresponding coupling term becomes
% \[
% \sum_{j=1}^{L}\pi_{ij}
% \left[
% V\big(x+E_{j}(i)x+F_{j}(i)u,j\big)-V(x,i)
% \right],
% \]
% which involves both a change of regime and a jump of the state. Therefore, the Riccati equations are expected to contain additional nonlinear terms involving the jump matrices $E_{j}(i)$ and $F_{j}(i)$. In the infinite horizon case, these terms also affect stabilizability, admissibility, convexity of the cost functional, and the existence of stabilizing solutions to the associated CAREs.

From the application point of view, the regime-switching jump-diffusion formulation is motivated by the observation that, in many real systems, a regime switch represents not only a change in model parameters but also an instantaneous shock to the state. While classical Markov-modulated diffusion models can describe different operating modes (see, for example, \cite{chen_stochastic_2022,WangJinWei2019}), they do not capture the sudden loss or gain occurring at the switching instant. By allowing regime transitions to generate endogenous jumps in the controlled state, the proposed model provides a more realistic framework for systems in which structural changes and abrupt shocks coexist. This modeling feature proves to be particularly advantageous in many application areas.
In financial markets, Markov regime-switching jump-diffusion models provide a natural framework for describing asset prices subject to both regime-dependent market parameters and abrupt jump shocks, and have been applied to portfolio selection and related optimal investment problems \cite{shi2025optimal,Zhang-Siu-Meng-2010,Zhang-Elliott-Siu-Guo-2011}. 
In stochastic control systems, the stochastic maximum principle for Markov regime-switching jump-diffusion systems has received widespread attention in the literature (see, for example, \cite{zhang_stochastic_2012,sun2018maximum,lv_stochastic_2018,song2022general,chen2023general,nguyen_general_2021}).
 Further studies on Markov regime-switching jump-diffusion systems in areas including risk-sensitive control, queueing theory, and statistical modeling can be found in Sun et al. \cite{sun_risk-sensitive_2018}, Azam et al. \cite{A} and Shu et al. \cite{B}.

In this paper, we will investigate the closed-loop solvability of Problem (M-SLQ)$_{\infty}$, and study a class of lifetime wealth tracking problem based on the derived results. The main contributions of this paper can be concluded as follows:
\begin{itemize}
  \item [(i)] We provide an equivalent characterization of the $L^{2}$ -stability for linear systems driven jointly by Brownian motion and Markovian jumps over the infinite time horizon, which further develops the corresponding result in \cite{WuLiZhang2025} by introducing Markovian jumps into the linear system.
  \item [(ii)] We present an equivalent characterization of the closed-loop solvability of Problem (M-SLQ)$_{\infty}$. This characterization allows us to obtain the optimal control in feedback form by solving the stabilizing solution of the corresponding CAREs.
  \item [(iii)] We study a lifetime wealth tracking problem and provide the corresponding optimal investment strategy along with its economic interpretation, thereby shedding light on the derived results.
\end{itemize}

The rest of this paper is organized as follows. Section \ref{sec-2} investigates the $L^{2}$-stability of the linear regime-switching jump-diffusion system and provides some  preliminary results. In section \ref{sec-3}, we establish the equivalence between the closed-loop solvability of Problem (M-SLQ)$_{\infty}$ and the existence of a static stabilizing solution to a class of CAREs. As an application, Section \ref{sec-4} studies a class of lifetime wealth tracking problem, whose numerical analysis is provided in Section \ref{sec-5}.

\section{Preliminary results}\label{sec-2}
We begin this section by introducing some notations that were not presented in the previous section. For any $M, N \in \mathbb{S}^n$, we write $M \geqslant N$ (respectively, $M>N$) if $M-N$ is semi-positive definite (respectively, positive definite). For any Banach space $\mathbb{B}$, we denote
$$\mathcal{D}\left(\mathbb{B}\right)\triangleq\left\{\mathbf{\Lambda}=\left(\Lambda(1),\cdots,\Lambda(L)\right) \mid \Lambda(i) \in \mathbb{B}\text{, } \forall i\in \mathcal{S}\right\}.$$
Specially, if $\mathbf{\Lambda}\in \mathcal{D}\left(\mathbb{S}^{n}\right)$, then we say $\lambda$ (resp. $\mu$) is the smallest (resp. largest) eigenvalue of $\mathbf{\Lambda}$ if it satisfies
$$
  \lambda=\min\{\lambda_{1},\lambda_{2},\cdots,\lambda_{L}\}\quad
\left(\mu=\max\{\mu_{1},\mu_{2},\cdots,\mu_{L}\}\right),
  $$
where $\lambda_{i}$ (resp. $\mu_{i}$) is the smallest (resp. largest) eigenvalue of $\Lambda(i)$,
$i\in\mathcal{S}$. Without causing confusion, we sometimes also say that $\lambda$ (resp. $\mu$) is the smallest (resp. largest) eigenvalue of process $\Lambda(\alpha)$.
Let $\mathcal{P}$ be the $\mathbb{F}$ predictable $\sigma$-field on $[0,\infty)\times\Omega$ and write $\varphi \in \mathbb{F}$ (resp. $\varphi \in \mathcal{P}$) if it is $\mathbb{F}$-progressively measurable (resp. $\mathcal{P}$-measurable). Then, for any Euclidean space $\mathbb{H}$, we introduce the following space:
$$
\begin{aligned}
&L_{\mathbb{F}}^{2,loc}(\mathbb{H}) =\left\{\varphi:[0, \infty)\times \Omega \rightarrow \mathbb{H} \mid \varphi(\cdot) \in \mathbb{F}\text{, } \mathbb{E} \int_{0}^{T}|\varphi(s)|^{2} ds<\infty, \quad\forall T>0\right\}, \\
&L_{\mathbb{F}}^{2}(\mathbb{H}) =\left\{\varphi:[0, \infty) \times \Omega \rightarrow \mathbb{H} \mid \varphi(\cdot) \in \mathbb{F}\text{, } \mathbb{E} \int_{0}^{\infty}|\varphi(s)|^{2} ds<\infty\right\},\\
&L_{\mathcal{P}}^{2}(\mathbb{H}) =\left\{\varphi:[0, \infty) \times \Omega \rightarrow \mathbb{H} \mid \varphi(\cdot) \in \mathcal{P}\text{, } \mathbb{E} \int_{0}^{\infty}|\varphi(s)|^{2} ds<\infty\right\}.
\end{aligned}
$$

Next, we introduce some concepts and useful results,  which will play an important role in our further analysis. Before we embark on this program, we first consider the following uncontrolled system:
\begin{equation}\label{AE}
\left\{
\begin{aligned}
&dX(t)=A(\alpha_{t})X(t)dt+C(\alpha_{t})X(t)dW(t)+\sum_{j=1}^{L}E_{j}(\alpha_{t-})X(t-)d\widetilde{N}_{j} (t), \quad t\geq 0,\\
&X(0)=x,\quad \alpha_{0}=i.
\end{aligned}
\right.
\end{equation}
For simplicity, we denote the above system as $[A,C,\mathbf{E}]_{\alpha}$ and introduce the following definition.
\begin{definition}
System $[A,C,\mathbf{E}]_{\alpha}$ is said to be $L^{2}$-stable if, for any $(x,i)\in\mathbb{R}^{n}\times\mathcal{S}$, the solution $X(\cdot)\equiv X(\cdot;x,i)$ of \eqref{AE}  is in $L_{\mathbb{F}}^{2}(\mathbb{R}^{n})$.
\end{definition}
The following proposition provides an equivalent characterization for the $L^{2}$-stability of system $[A,C,\mathbf{E}]_{\alpha}$.

\begin{proposition}\label{prop-L2}
  The system $[A,C,\mathbf{E}]_{\alpha}$  is $L^{2}$-stable if and only if there exist a $\mathbf{P}=\left[P(1),P(2),...,P(L)\right]\in\mathcal{D}\left(\mathbb{S}_{+}^{n}\right)$ such that:
    \begin{equation}\label{eq-L2}
    \begin{aligned}
    P(i)A(i)&+A(i)^{\top}P(i)+C(i)^{\top}P(i)C(i)+\sum_{j\neq i}^{L}\pi_{ij}\left[(P(j)-P(i))E_{j}(i)\right.\\
    &\left.+E_{j}(i)^{\top}(P(j)-P(i))+E_{j}(i)^{\top}P(j)E_{j}(i)+P(j)-P(i)\right]< 0,\quad i\in\mathcal{S}.
    \end{aligned}
    \end{equation}
\end{proposition}
\begin{proof} The proof of sufficiency can be considered as a special case of Theorem \ref{thm-FSDE} with $b(\cdot)=\sigma(\cdot)=\gamma_{j}(\cdot)=0$. Here, we only prove the necessity. By the definition, the $L^{2}$-stability of system $[A,C,\mathbf{E}]_{\alpha}$ implies that
  \begin{equation}\label{prop-L2-1}
  \mathbb{E}\Big[\int_{0}^{\infty}\big<\Lambda(\alpha_{t})X(t;x,i),X(t;x,i)\big>dt\Big]<\infty,\qquad \forall (x,i)\in\mathbb{R}^{n}\times\mathcal{S},
  \end{equation}
for any given $\mathbf{\Lambda}=\left[\Lambda(1),\Lambda(2),...,\Lambda(L)\right]\in\mathcal{D}\left(\mathbb{S}_{+}^{n}\right)$.
  
  Now, let us consider the following linear equations:
  \begin{equation}\label{prop-L2-2}
  \left\{
  \begin{aligned}
   & \Dot{P}(t,i)=P(t,i)A(i)+A(i)^{\top}P(t,i)+C(i)^{\top}P(t,i)C(i)+\Lambda(i)+\sum_{j\neq i}^{L}\pi_{ij}\big[(P(t,j)-P(t,i))E_{j}(i)\\
    &\qquad\qquad+E_{j}(i)^{\top}(P(t,j)-P(t,i))+E_{j}(i)^{\top}P(t,j)E_{j}(i)+P(t,j)-P(t,i)\big],\\
   &P(0,i)=0, \quad i\in\mathcal{S},\quad t\geq 0,
   \end{aligned}
   \right.
  \end{equation}
  which admits a unique solution $\left\{P(\cdot,i)\right\}_{i\in\mathcal{S}}$ defined on the interval $[0,\infty)$. For any fixed $\tau>0$, we define
  $$P^{\tau}(t,i)\triangleq P(\tau-t,i),\quad t\in[0,\tau],\quad i\in\mathcal{S}.$$
  Then $\left\{P^{\tau}(\cdot,i)\right\}_{i\in\mathcal{S}}$ is the solution to the following equations:
  \begin{equation*}
  \left\{
  \begin{aligned}
       &\Dot{P}^{\tau}(t,i)=-\big\{P^{\tau}(t,i)A(i)+A(i)^{\top}P^{\tau}(t,i)+C(i)^{\top}P^{\tau}(t,i)C(i)+\Lambda(i)+\sum_{j\neq i}^{L}\pi_{ij}\big[(P^{\tau}(t,j)-P^{\tau}(t,i))E_{j}(i)\\
    &\qquad\qquad+E_{j}(i)^{\top}(P^{\tau}(t,j)-P^{\tau}(t,i))+E_{j}(i)^{\top}P^{\tau}(t,j)E_{j}(i)+P^{\tau}(t,j)-P^{\tau}(t,i)\big]\big\},\\
   &P^{\tau}(\tau,i)=0, \quad t\in[0,\tau], \quad i\in\mathcal{S}.
   \end{aligned}
   \right.
  \end{equation*}
  
  Note that for any $(x,i)\in\mathbb{R}^{n}\times\mathcal{S}$, 
  \begin{equation*}
  \begin{aligned}
 &\quad -\left<P(\tau,i)x,x\right>= -\left<P^{\tau}(0,i)x,x\right>=\mathbb{E}\big[\left<P^{\tau}(\tau,\alpha_{\tau})X(\tau;x,i),X(\tau;x,i)\right>
 -\left<P^{\tau}(0,i)x,x\right>\big]\\
 &=\mathbb{E}\int_{0}^{\tau}\Big<\Big\{\Dot{P}^{\tau}(t,\alpha_{t})+P^{\tau}(t,\alpha_{t})A(\alpha_{t})
 +A(\alpha_{t})^{\top}P^{\tau}(t,\alpha_{t})+C(\alpha_{t})^{\top}P^{\tau}(t,\alpha_{t})C(\alpha_{t})
 \\
  &\qquad\qquad +\sum_{j=1}^{L}\lambda_{j}(t)\Big[P^{\tau}(t,j)-P^{\tau}(t,\alpha_{t})
  +(P^{\tau}(t,j)-P^{\tau}(t,\alpha_{t}))E_{j}(\alpha_{t})
  \\
  &\qquad\qquad +E_{j}(\alpha_{t})^{\top}(P^{\tau}(t,j)-P^{\tau}(t,\alpha_{t})) +E_{j}(\alpha_{t})^{\top}P^{\tau}(t,j)E_{j}(\alpha_{t}) \Big]\Big\}X(t;x,i),X(t;x,i)\Big>dt\\
  &=-\mathbb{E}\int_{0}^{\tau}\big<\Lambda(\alpha_{t}) X(t;x,i),X(t;x,i)\big>dt
  =-\mathbb{E}\Big\{\int_{0}^{\tau}\big<\Phi_{i}(t)^{\top}\Lambda(\alpha_{t})\Phi_{i}(t) x,x\big>dt\mid\alpha_{0}=i\Big\},
  \end{aligned}
  \end{equation*}
  where $\Phi_{i}(\cdot)$ satisfies
  \begin{equation}\label{AC-Phi}
\left\{
\begin{aligned}
&d\Phi_{i}(t)=A(\alpha_{t})\Phi_{i}(t)dt+C(\alpha_{t})\Phi_{i}(t)dW(t)+\sum_{j=1}^{L}E_{j}(\alpha_{t-})\Phi_{i}(t-)d\widetilde{N}_{j}(t), \quad t\geq 0,\\
&\Phi_{i}(0)=I,\quad \alpha_{0}=i.
\end{aligned}
\right.
\end{equation}
Following from arbitrary $x\in\mathbb{R}^{n}$, we obtain that the solution of equations \eqref{prop-L2-2} admits the following representation:
$$P(\tau,i)=\mathbb{E}\Big\{\int_{0}^{\tau}\Phi_{i}(t)^{\top}\Lambda(\alpha_{t})\Phi_{i}(t)dt\mid \alpha_{0}=i\Big\},\quad \tau\geq 0.$$

Let $\mathbf{\Lambda}=\left[\Lambda(1),\Lambda(2),...,\Lambda(L)\right]\in\mathcal{D}\left(\mathbb{S}_{+}^{n}\right)$.
Then $P(\tau,i)$ defined above is increasing in $\tau$. Additionally, by equation \eqref{prop-L2-1}, we can further obtain that
$$\mathbb{E}\Big\{\int_{0}^{\infty}\Phi_{i}(t)^{\top}\Lambda(\alpha_{t})\Phi_{i}(t)dt\mid \alpha_{0}=i\Big\}<\infty,\quad i\in\mathcal{S}.$$
Hence, by monotone convergence theorem,  there exists a $P(i)\in\mathbb{S}_{+}^{n}$ such that
$$P(i)=\lim_{\tau\rightarrow\infty}P(\tau,i)=\mathbb{E}\Big\{\int_{0}^{\infty}\Phi_{i}(t)^{\top}\Lambda(\alpha_{t})\Phi_{i}(t)dt\mid \alpha_{0}=i\Big\},\qquad  i\in\mathcal{S}.$$

On the other hand, note that
$$
\begin{aligned}
& P(\tau+1,i)-P(\tau,i)=\int_{\tau}^{\tau+1}\big\{P(t,i)A(i)+A(i)^{\top}P(t,i)+C(i)^{\top}P(t,i)C(i)+\Lambda(i)\\
&\quad+\sum_{j\neq i}^{L}\pi_{ij}\big[(P(t,j)-P(t,i))E_{j}(i)+E_{j}(i)^{\top}(P(t,j)-P(t,i))+E_{j}(i)^{\top}P(t,j)E_{j}(i)+P(t,j)-P(t,i)\big]\big\}dt.
\end{aligned}$$
Letting $\tau\rightarrow\infty$ yields
$$
\begin{aligned}
0&=P(i)A(i)+A(i)^{\top}P(i)+C(i)^{\top}P(i)C(i)+\Lambda(i)+\sum_{j\neq i}^{L}\pi_{ij}\big[(P(j)-P(i))E_{j}(i)\\
&\quad+E_{j}(i)^{\top}(P(j)-P(i))+E_{j}(i)^{\top}P(j)E_{j}(i)+P(j)-P(i)\big],
\end{aligned}
$$
which, combining with the fact $\mathbf{\Lambda}\in\mathcal{D}\left(\mathbb{S}_{+}^{n}\right)$, further implies that the condition \eqref{eq-L2} holds.
\end{proof}

\begin{remark}\label{rmk-stable}\rm
From the above proposition, one can easily obtain a sufficient condition to guarantee the  $L^{2}$-stability of the system $[A,C,\mathbf{E}]_{\alpha}$, that is,
$$
A(i)+A(i)^{\top}+C(i)^{\top}C(i)+\sum_{j\neq i}^{L}\pi_{ij}E_{j}(i)^{\top}E_{j}(i)< 0,\quad i,j\in\mathcal{S}.
$$
In this case, we can  set $P(i)=I$ ($i\in\mathcal{S}$) so that inequality \eqref{eq-L2} holds.
\end{remark}

Next, we consider the following linear SDE in the infinite time horizon $[0,\infty)$:
 \begin{equation}\label{FSDE}
\left\{
\begin{aligned}
&dX(t)=\left[A(\alpha_{t})X(t)+b(t)\right]dt+\left[C(\alpha_{t})X(t)+\sigma(t)\right]dW(t)
+\sum_{j=1}^{L}\left[E_{j}(\alpha_{t-})X(t-)+\gamma_{j}(t)\right]d\widetilde{N}_{j}(t),\\
&X(0)=x,\quad \alpha_{0}=i,\quad t\geq 0.
\end{aligned}
\right.
\end{equation}

\begin{definition} A $X(\cdot)$ is called the $L^{2}$-stable adapted solution of \eqref{FSDE} if  $X(\cdot)$ is a solution of SDE
 \eqref{FSDE} and satisfies  $X(\cdot)\in L_{\mathbb{F}}^{2}(\mathbb{R}^{n})$.

\end{definition}

\begin{theorem}\label{thm-FSDE}
  Suppose that system $[A,C,\mathbf{E}]_{\alpha}$ is $L^{2}$-stable. Then for any given $b(\cdot),\,\sigma(\cdot)\in L_{\mathbb{F}}^{2}(\mathbb{R}^{n})$, $\gamma_{j}(\cdot)\in L_{\mathcal{P}}^{2}(\mathbb{R}^{n}) \,(j\in\mathcal{S})$ and initial value $(x,i)\in\mathbb{R}^{n}\times\mathcal{S}$,  the FSDE \eqref{FSDE} admits a unique $L^{2}$-stable adapted solution $X(\cdot)$  such that
  \begin{equation}\label{FSDE-E}
   \mathbb{E}\int_{0}^{\infty}|X(t)|^{2}dt\leq K\Big[ |x|^{2}+\mathbb{E}\int_{0}^{\infty}\big[|b(t)|^{2}+|\sigma(t)|^{2}+\sum_{j=1}^{L}\lambda_{j}(t) |\gamma_{j}(t)|^{2}\big]dt\Big],\quad \text{for some }K>0.
  \end{equation}
\end{theorem}

\begin{proof}
The uniqueness is clearly established. For any given $(x,i)\in\mathbb{R}^{n}\times\mathcal{S}$, since $[A,C,\mathbf{E}]_{\alpha}$ is $L^{2}$-stable, by Proposition \ref{prop-L2}, there exists a $\mathbf{P}=\left[P(1),P(2),...,P(L)\right]\in\mathcal{D}\left(\mathbb{S}_{+}^{n}\right)$ such that
 $$
 \begin{aligned}
 \Lambda(i)\triangleq & -\Big\{P(i)A(i)+A(i)^{\top}P(i)+C(i)^{\top}P(i)C(i)+\sum_{j\neq i}^{L}\pi_{ij}\big[(P(j)-P(i))E_{j}(i)+E_{j}(i)^{\top}(P(j)-P(i))\\
    &+E_{j}(i)^{\top}P(j)E_{j}(i)+P(j)-P(i)\big]\Big\}> 0,\quad i\in\mathcal{S}.
 \end{aligned}
 $$
Applying It\^o's rule to $\left<P(\alpha_{t})X(t),X(t)\right>$, we obtain
\begin{align*}
&\quad\mathbb{E}\left[\left<P(\alpha_{t})X(t),X(t)\right>-\left<P(i)x,x\right>\right]\\
%&=\mathbb{E}\int_{0}^{t}\Big[\big<\big[P(\alpha_{s})A(\alpha_{s})+A(\alpha_{s})^{\top}P(\alpha_{s})
%   +\sum_{j=1}^{L}\lambda_{j}(s)\big((P(j)-P(\alpha_{s}))E_{j}(\alpha_{s})\\
%   &\quad+E_{j}(\alpha_{s})^{\top}(P(j)-P(\alpha_{s}))+E_{j}(\alpha_{s})^{\top}P(j)+E_{j}(\alpha_{s})+P(j)-P(\alpha_{s})\big) \big]X(s),X(s)\big>\\
%&\quad+2\big<P(\alpha_{s})b(s)+\sum_{j=1}^{L}\lambda_{j}(s)\big[(P(j)-P(\alpha_{s}))\gamma_{j}(s)+E_{j}(\alpha_{s})^{\top}P(j)\gamma_{j}(s)\big],X(s)\big>\\
% &\quad  +\sum_{j=1}^{L}\lambda_{j}(s)\big<P(j)\gamma_{j}(s),\gamma_{j}(s)\big>\Big]ds\\
&=\mathbb{E}\int_{0}^{t}\Big[-\big<\Lambda(\alpha_{s}) X(s),X(s)\big>+\big<P(\alpha_{s})\sigma(s),\sigma(s)\big>+\sum_{j=1}^{L}\lambda_{j}(s)\big<P(j)\gamma_{j}(s),\gamma_{j}(s)\big>\\
&\quad+2\big<P(\alpha_{s})b(s)+C(\alpha_{s})^{\top}P(\alpha_{s})\sigma(s)+\sum_{j=1}^{L}\lambda_{j}(s)\big[(P(j)-P(\alpha_{s}))\gamma_{j}(s)
+E_{j}(\alpha_{s})^{\top}P(j)\gamma_{j}(s)\big],X(s)\big>\Big]ds.
\end{align*} 
Let
\begin{align*}
&\Psi(i)\triangleq P(i)^{-\frac{1}{2}}\Lambda(i)P(i)^{-\frac{1}{2}}>0, \\
&\eta(s)\triangleq P(\alpha_{s})^{-\frac{1}{2}}\Big[P(\alpha_{s})b(s)+C(\alpha_{s})^{\top}P(\alpha_{s})\sigma(s)+\sum_{j=1}^{L}\lambda_{j}(s)\big[(P(j)-P(\alpha_{s}))\gamma_{j}(s)
+E_{j}(\alpha_{s})^{\top}P(j)\gamma_{j}(s)\big]\Big].
\end{align*}
Then
\begin{align*}
&\quad\frac{d}{dt}\mathbb{E}\left[|P(\alpha_{t})^{\frac{1}{2}}X(t)|^{2}\right]
=\frac{d}{dt}\mathbb{E}\left[\left<P(\alpha_{t})X(t),X(t)\right>\right]\\
&=-\mathbb{E}\big[\big<\Psi(\alpha_{t})P(\alpha_{t})^{\frac{1}{2}}X(t),P(\alpha_{t})^{\frac{1}{2}}X(t)\big>\big]
+2\mathbb{E}\big[\big<\eta(t),P(\alpha_{t})^{\frac{1}{2}}X(t)\big>\big]
+\mathbb{E}\big[\big<P(\alpha_{t})\sigma(t),\sigma(t)\big>\big]\\
&\quad+\mathbb{E}\big[\sum_{j=1}^{L}\lambda_{j}(t)\big<P(j)\gamma_{j}(t),\gamma_{j}(t)\big>\big]\\
&\leq -\delta\mathbb{E}\big[|P(\alpha_{t})^{\frac{1}{2}}X(t)|^{2}\big]
+\frac{\delta}{2}\mathbb{E}\big[|P(\alpha_{t})^{\frac{1}{2}}X(t)|^{2}\big]
+\frac{2}{\delta}\mathbb{E}\big[|\eta(t)|^2\big]+\mu\mathbb{E}\big[\sum_{j=1}^{L}\lambda_{j}(t)|\gamma_{j}(t)|^2+|\sigma(t)|^2\big]\\
&=-\frac{\delta}{2}\mathbb{E}\big[|P(\alpha_{t})^{\frac{1}{2}}X(t)|^{2}\big]
+\frac{2}{\delta}\mathbb{E}\big[|\eta(t)|^2\big]+\mu\mathbb{E}\big[\sum_{j=1}^{L}\lambda_{j}(t)|\gamma_{j}(t)|^2+|\sigma(t)|^2\big],
\end{align*}
where $\delta >0$ is the smallest  eigenvalue of $\mathbf{\Psi}=\left[\Psi(1),\Psi(2),...,\Psi(L)\right]$ and $\mu>0$ is the largest eigenvalue of $\mathbf{P}=\left[P(1),P(2),...,P(L)\right]$.
Let $$\phi(t)\triangleq \mathbb{E}\left[|P(\alpha_{t})^{\frac{1}{2}}X(t)|^{2}\right], \quad
\beta(t)\triangleq \frac{2}{\delta}\mathbb{E}\big[|\eta(t)|^2\big]+\mu\mathbb{E}\big[\sum_{j=1}^{L}\lambda_{j}(t)|\gamma_{j}(t)|^2+|\sigma(t)|^2\big].$$
Then,
$$d\phi(t)\leq \left[-\frac{\delta}{2}\phi(t)+\beta(t)\right]dt,$$
which implies
$$\frac{d}{dt}\left[e^{\frac{\delta}{2}t}\phi(t)\right]
=e^{\frac{\delta}{2}t}\left[\frac{\delta}{2}\phi(t)+\phi'(t)\right]
\leq e^{\frac{\delta}{2}t}\beta(t).$$
Hence,
$$\phi(t) \leq \phi(0)e^{-\frac{\delta}{2}t}+\int_{0}^{t}e^{-\frac{\delta}{2}(t-s)}\beta(s)ds.$$
Integrating both sides of the above equation simultaneously yields
$$
\begin{aligned}
\int_{0}^{\infty}\phi(t)dt
&\leq\int_{0}^{\infty}\left[\phi(0)e^{-\frac{\delta}{2}t}+\int_{0}^{t}e^{-\frac{\delta}{2}(t-s)}\beta(s)ds\right]dt\\
&=\frac{2}{\delta}\phi(0)+\int_{0}^{\infty}\int_{s}^{\infty}e^{-\frac{\delta}{2}(t-s)}\beta(s)dtds\\
&=\frac{2}{\delta}\phi(0)+\frac{2}{\delta}\int_{0}^{\infty}\beta(s)ds.
\end{aligned}
$$
Therefore
$$
\begin{aligned}
&\quad\mathbb{E}\int_{0}^{\infty}|P(\alpha_{t})^{\frac{1}{2}}X(t)|^{2}dt=\int_{0}^{\infty}\phi(t)dt
\leq\frac{2}{\delta}\phi(0)+\frac{2}{\delta}\int_{0}^{\infty}\beta(t)dt\\
&=\frac{2}{\delta}|P(i)^{\frac{1}{2}}x|^{2}
+\frac{2}{\delta}\mathbb{E}\int_{0}^{\infty}\left[\frac{2}{\delta}|\eta(t)|^2+\mu\big(\sum_{j=1}^{L}\lambda_{j}(t)|\gamma_{j}(t)|^2+|\sigma(t)|^2\big)\right]dt.
\end{aligned}
$$
Noting that
$$|\eta(t)|^2\leq M\left[|b(t)|^2+|\sigma(t)|^2+\sum_{j=1}^{L}\lambda_{j}(t)|\gamma_{j}(t)|^2\right],\qquad \text{for some } M>0,$$
The desired result can be obtained from the above equation.
\end{proof}

In the end of this section, we introduce the concept of pseudo-inverse (see Penrose \cite{Penrose.1955}). For any $M\in\mathbb{R}^{m\times n}$, there exists a unique matrix $M^{\dag
}\in\mathbb{R}^{n\times m}$, called the pseudo-inverse of M, such that
$$MM^{\dag}M=M,\qquad M^{\dag}MM^{\dag}=M^{\dag},\qquad (MM^{\dag})^{\top}=MM^{\dag},\qquad 
(M^{\dag}M)^{\top}=M^{\dag}M.$$
In addition, if $M\in\mathbb{S}^{n}$, then $M^{\dag}\in\mathbb{S}^{n}$, and 
$$MM^{\dag}=M^{\dag}M,\qquad M\geq 0\Leftrightarrow M^{\dag}\geq 0.$$

\begin{lemma}[Extended Schur's lemma \cite{Albert1.969}]\label{lem-Schur}
Let $M\in\mathbb{S}^{n}$, $N\in\mathbb{S}^{m}$, $L\in\mathbb{R}^{n\times m}$. Then the following conditions are equivalent:
\begin{description}
  \item[(i)] $M-LN^{\dag}L^{\top}\geq 0$, $N\geq 0$, and $L(I-NN^{\dag})=0$.
  \item[(ii)] $\left(\begin{array}{cc}
       M &  L\\
       L^{\top}& N 
  \end{array}\right)\geq 0$.
\end{description}
\end{lemma}

%\begin{remark}\rm
%We point out that the condition $L(I-NN^{\dag})=0$ is equivalent to $\mathcal{R}(L^{\top})\subseteq \mathcal{R}(N)$, where $\mathcal{R}(\Lambda)$ represents the range of matrix $\Lambda$.
%\end{remark}

\section{Main results}\label{sec-3}
In this section, we study the closed-loop solvability for Problem (M-SLQ)$_{\infty}$.  For simplicity, we specify that system $\left[A,C,\mathbf{E};B,D,\mathbf{F}\right]_{\alpha}$ represents the controlled SDE \eqref{intro-state}. An element  
$
\mathbf{\Theta}=\left[\Theta(1),\Theta(2),...,\Theta(L)\right]\in\mathcal{D}\left(\mathbb{R}^{m\times n}\right)
$
is called a %$L^{2}$-stabilizer 
closed-loop strategy of system $\left[A,C,\mathbf{E};B,D,\mathbf{F}\right]_{\alpha}$ if and only if the following closed-loop system admits a unique solution $X(\cdot)\equiv X(\cdot;x,i,\mathbf{\Theta})\in L_{\mathbb{F}}^{2}(\mathbb{R}^{n})$ for any given initial value $(x,i)\in \mathbb{R}^{n}\times\mathcal{S}$:
\begin{equation}\label{state-closed}
  \left\{
 \begin{aligned}
   dX(t)&=\left[A\left(\alpha_{t}\right)+B\left(\alpha_{t}\right)\Theta\left(\alpha_{t}\right)\right]X(t)dt
   +\left[C\left(\alpha_{t}\right)+D\left(\alpha_{t}\right)\Theta\left(\alpha_{t}\right)\right]X(t)dW(t)\\
   &\quad+\sum_{j=1}^{L}\left[E_{j}\left(\alpha_{t-}\right)+F_{j}\left(\alpha_{t-}\right)\Theta\left(\alpha_{t-}\right)\right]X(t-)d\widetilde{N}_{j}(t),\qquad t\geq0,\\
   X(0)&=x,\quad \alpha(0)=i.
   \end{aligned}
  \right.
\end{equation}
In the following, we denote $\mathcal{H}_{\alpha}\equiv \mathcal{H}\left[A,C,\mathbf{E};B,D,\mathbf{F}\right]_{\alpha}$ as the set containing all %$L^{2}$-stabilizers 
closed-loop strategy of system $\left[A,C,\mathbf{E};B,D,\mathbf{F}\right]_{\alpha}$, and define 
$$
u(t;x,i,\mathbf{\Theta})=\Theta(\alpha_{t-})X(t;x,i,\mathbf{\Theta}),\qquad t\geq 0,
$$
as the outcome of closed-loop strategy $\mathbf{\Theta}\in \mathcal{H}_{\alpha}$.
%called the system $\left[A,C,\mathbf{E};B,D,\mathbf{F}\right]_{\alpha}$ is $L^{2}$-stabilizable if $\mathcal{H}_{\alpha}\neq\emptyset$. 
Now, we introduce the concept of closed-loop solvability  for Problem (M-SLQ)$_{\infty}$ as follows.

%and define
%$$\mathcal{H}_{\alpha}\triangleq
%\left\{\mathbf{\Theta}=\left[\Theta(1),\Theta(2),...,\Theta(L)\right]\in\mathcal{D}\left(\mathbb{R}^{m\times n}\right)\Big|\left[A+B\Theta,\mathbf{E}+\mathbf{F}\Theta\right]_{\alpha} \text{ is } L^{2}\text{- stable}.\right\}.$$
%
%With the above notations, we introduce the following definitions.
%
%\begin{definition}
%The system $\left[A,\mathbf{E};B,\mathbf{F}\right]_{\alpha}$ is said to be $L^{2}$-stabilizable if and only if $\mathcal{H}_{\alpha}\neq\emptyset.$ Any element $\Theta\in\mathcal{H}_{\alpha}$ is called a stabilizer of system $\left[A,\mathbf{E};B,\mathbf{F}\right]_{\alpha}$.
%\end{definition}

\begin{definition}
\begin{description}
  \item[(i)] An element  $\widehat{\mathbf{\Theta}}\in\mathcal{H}_{\alpha}$ is called a  closed-loop optimal strategy of Problem (M-SLQ) if for any $\left(x,i\right)\in \mathbb{R}^{n}\times\mathcal{S}$, the following holds:
  \begin{equation}\label{closed}
   J(x,i;u(\cdot;x,i,\widehat{\mathbf{\Theta}}))\leq J(x,i;u(\cdot)), \quad \forall u(\cdot)\in  \mathcal{U}_{ad}(x,i).
  \end{equation}
  \item[(ii)] Problem (M-SLQ)$_{\infty}$ is said to be (uniquely) closed-loop solvable if it has a (unique)  closed-loop optimal strategy $\widehat{\mathbf{\Theta}}\in\mathcal{H}_{\alpha}$.
\end{description}
\end{definition}

In the rest of the paper, we suppose the following assumption always holds.

\textbf{(H1).} The closed-loop strategy set is non-empty, i.e, $\mathcal{H}_{\alpha}\neq\emptyset$.

\begin{remark}\rm
Obviously, if system $\left[A,C,\mathbf{E}\right]_{\alpha}$ is $L^{2}$-stable, then $\mathbf{0}\in\mathcal{H}_{\alpha}$, which implies that the closed-loop strategy set is non-empty. 
On the other hand, for any $\mathbf{\Theta}\in \mathcal{H}_{\alpha}$, Theorem \ref{thm-FSDE} implies that the outcome $u(\cdot;x,i,\mathbf{\Theta})$ of closed-loop strategy $\mathbf{\Theta}$ is in $\mathcal{U}_{ad}(x,i)$. Consequently, the assumption (H1) further implies that the set $\mathcal{U}_{ad}(x,i)$ is nonempty for any $(x,i)$. % The following proposition provides a more profound characterization of $\mathcal{U}_{ad}(x,i)$.
\end{remark}

To simplify our subsequent analysis, for any given $\mathbf{P}=\left[P(1),P(2),...,P(L)\right]\in\mathcal{D}\left(\mathbb{S}^{n}\right)$, we introduce the following notations
\begin{equation}\label{MLN}
    \begin{aligned}
    &\mathcal{M}(P,i)\triangleq P(i)A(i)+A(i)^{\top}P(i)+C(i)^{\top}P(i)C(i)+Q(i)+\sum_{j\neq i}^{L}\pi_{ij}\big[(P(j)-P(i))E_{j}(i)\\
    &\qquad\qquad+E_{j}(i)^{\top}(P(j)-P(i))+E_{j}(i)^{\top}P(j)E_{j}(i)+P(j)-P(i)\big],\\
    &\mathcal{L}(P,i)\triangleq
    P(i)B(i)+C(i)^{\top}P(i)D(i)+S(i)^{\top}+\sum_{j\neq i}^{L}\pi_{ij}\big[(P(j)-P(i))F_{j}(i)+E_{j}(i)^{\top}P(j)F_{j}(i)\big],\\
    &\mathcal{N}(P,i)\triangleq
    R(i)+D(i)^{\top}P(i)D(i)+\sum_{j\neq i}^{L}\pi_{ij}F_{j}(i)^{\top}P(j)F_{j}(i).
    \end{aligned}
\end{equation}
The following Lemma  provides us with another form of cost functional \eqref{intro-cost}.
\begin{lemma}\label{lem-useful}
 For any $\mathbf{P}=\left[P(1),P(2),...,P(L)\right]\in\mathcal{D}\left(\mathbb{S}^{n}\right)$,  the cost functional \eqref{intro-cost} admits the following representation:
 \begin{equation}\label{cost-useful-1}
\begin{aligned}
   &\quad J(x,i;u(\cdot))=\big<P(i)x,x\big>+\mathbb{E}
    \int_{0}^{\infty}
    \left<
    \left[
    \begin{array}{cc}
    \mathcal{M}(P,\alpha_{t})&\mathcal{L}(P,\alpha_{t})\\
    \mathcal{L}(P,\alpha_{t})^{\top}&\mathcal{N}(P,\alpha_{t})
    \end{array}
    \right]
    \left[
    \begin{array}{c}
    X(t)\\
    u(t)
    \end{array}
    \right],
    \left[
    \begin{array}{c}
    X(t)\\
    u(t)
    \end{array}
    \right]
    \right>dt,
  \end{aligned}
\end{equation}
where the $X(\cdot)$ is the solution of SDE \eqref{intro-state}.
\end{lemma}

\begin{proof}
By applying It\^o's formula to $\left<P(\alpha_{t})X(t),X(t)\right>$, we have
 \begin{equation}\label{useful}
 \begin{aligned}
   &\quad\mathbb{E}\left[\left<P(\alpha_{T})X(T),X(T)\right>-\left<P(i)x,x\right>\right]\\
   &=\mathbb{E}\int_{0}^{T}\Big\{
   \big<\big[\mathcal{M}(P,\alpha_{t})-Q(\alpha_{t})\big]X(t),X(t)\big>
   +2\big<\big[\mathcal{L}(P,\alpha_{t})-S(\alpha_{t})^{\top}\big]u(t),X(t)\big>\\
   &\qquad\qquad+\big<\big[\mathcal{N}(P,\alpha_{t})-R(\alpha_{t})\big]u(t),u(t)\big>\Big\}dt.
   \end{aligned}
 \end{equation}
 Noting that the $L^{2}$-stability of state process $X(\cdot)$ yields $\lim_{T\rightarrow +\infty}\mathbb{E}\left[\left<P(\alpha_{T})X(T),X(T)\right>\right]=0.$
Hence, letting $T\rightarrow +\infty$, we have 
$$
 \begin{aligned}
0&=\left<P(i)x,x\right>+\mathbb{E}\int_{0}^{\infty}\Big\{
   \big<\big[\mathcal{M}(P,\alpha_{t})-Q(\alpha_{t})\big]X(t),X(t)\big>
   +2\big<\big[\mathcal{L}(P,\alpha_{t})-S(\alpha_{t})^{\top}\big]u(t),X(t)\big>\\
   &\quad+\big<\big[\mathcal{N}(P,\alpha_{t})-R(\alpha_{t})\big]u(t),u(t)\big>\Big\}dt.
\end{aligned}$$
 The desired result follows substituting the above equation into the cost functional \eqref{intro-cost}. 
\end{proof}

\begin{remark}\label{rmk-useful}\rm
Let $\mathbf{\Theta}\in \mathcal{H}_{\alpha}$ be a closed-loop strategy, and $u(\cdot;x,i,\mathbf{\Theta})$ be its outcome control at initial pair $(x,i)\in\mathbb{R}^{n}\times\mathcal{S}$. Then one can similarly derives the following result:
\begin{equation}\label{cost-useful-2}
\begin{aligned}
   &J(x,i;u(\cdot;x,i,\mathbf{\Theta}))=\big<P(i)x,x\big>+\mathbb{E}
    \int_{0}^{\infty}
    \Big<
    \big[\mathcal{M}(P,\alpha_{t})+\mathcal{L}(P,\alpha_{t})\Theta(\alpha_{t})+\Theta(\alpha_{t})^{\top}\mathcal{L}(P,\alpha_{t})^{\top}\\
   &\quad +\Theta(\alpha_{t})^{\top}\mathcal{N}(P,\alpha_{t})\Theta(\alpha_{t})\big]X(t), X(t) \Big>dt.
  % &=\big<P(i)x,x\big>+\mathbb{E}
%    \int_{0}^{\infty}
%    \Big\{\Big<\big[\mathcal{M}(P,\alpha_{t})-\mathcal{L}(P,\alpha_{t})\mathcal{N}(P,\alpha_{t})^{\dag}\mathcal{L}(P,\alpha_{t})\big]X(t), X(t) \Big>\\
%    &\quad+\Big<\mathcal{N}(P,\alpha_{t})\big[\Theta(\alpha_{t})+\mathcal{N}(P,\alpha_{t})^{\dag}\mathcal{L}(P,\alpha_{t})^{\top}\big]X(t),
%    \big[\Theta(\alpha_{t})+\mathcal{N}(P,\alpha_{t})^{\dag}\mathcal{L}(P,\alpha_{t})^{\top}\big]X(t)\Big>\Big\}dt.
  \end{aligned}
\end{equation}
\end{remark}

In the following, we study the closed-loop solvability of Problem (M-SLQ)$_{\infty}$. We first consider the following CAREs:
\begin{equation}\label{CAREs-SLQ}
\left\{
   \begin{aligned}
    &\mathcal{M}(P,i)-\mathcal{L}(P,i) \mathcal{N}(P,i)^{\dag} \mathcal{L}(P,i)^{\top} = 0,\\
    & \mathcal{L}(P,i)\left[I-\mathcal{N}(P,i)\mathcal{N}(P,i)^{\dag}\right]=0,\\ 
    & \mathcal{N}(P,i)\geq 0,\quad\forall i\in\mathcal{S}.
   \end{aligned} 
   \right.
\end{equation}

\begin{definition}\label{def-stabilizing-solution-CAREs}
An element $\mathbf{P}\in\mathcal{D}\left(\mathbb{S}^{n}\right)$ is called a static stabilizing solution of \eqref{CAREs-SLQ} if $\mathbf{P}$ solves CAREs \eqref{CAREs-SLQ}, and there exists $\mathbf{\Pi}\in\mathcal{D}\left(\mathbb{R}^{m\times n}\right)$ such that
\begin{equation}\label{CAREs-ZLQ-stabilizer}
\mathcal{K}(\mathbf{\Pi})=\left(\mathcal{K}\left(\Pi(1)\right),\mathcal{K}\left(\Pi(2)\right),\cdots,\mathcal{K}\left(\Pi(L)\right)\right)\in\mathcal{H}_{\alpha},
\end{equation}
where 
$$\mathcal{K}\left(\Pi(i)\right)\triangleq-\mathcal{N}(P,i)^{\dag}\mathcal{L}(P,i)^{\top}+\left[I-\mathcal{N}(P,i)^{\dag}\mathcal{N}(P,i)\right]\Pi(i),\qquad i\in\mathcal{S}.$$
\end{definition}

The following theorem provides an equivalent characterization for the closed-loop solvable of Problem (M-SLQ).

\begin{theorem}\label{thm-SLQ_0-closed}
The Problem (M-SLQ)$_{\infty}$ is closed-loop solvable if and only if the CAREs \eqref{CAREs-SLQ} admits a static stabilizing solution $\mathbf{P}=\left[P(1),P(2),...,P(L)\right]\in\mathcal{D}\left(\mathbb{S}^{n}\right)$. 
In this case, the closed-loop optimal strategy $\widehat{\mathbf{\Theta}}=\left[\widehat{\Theta}(1),\widehat{\Theta}(2),\cdots,\widehat{\Theta}(L)\right]\in\mathcal{H}_{\alpha}$ admits the following representation:
\begin{equation}\label{eq-Theta}
  \widehat{\Theta}(i)=-\mathcal{N}(P,i)^{\dag}\mathcal{L}(P,i)^{\top}+\big[I-\mathcal{N}(P,i)^{\dag}\mathcal{N}(P,i)\big]\Pi(i),\quad \forall i\in\mathcal{S},
\end{equation}
where $\mathbf{\Pi}=\left(\Pi(1),\Pi(2),\cdots,\Pi(L)\right)\in\mathcal{D}\left(\mathbb{R}^{m\times n}\right)$ is chosen such that $\widehat{\mathbf{\Theta}}\in\mathcal{H}_{\alpha}$. Moreover, the value function is given by
\begin{equation}\label{eq-value-function}
      V(x,i)=\big<P(i)x,x\big>.
\end{equation}
\end{theorem}

\begin{proof}
\textbf{Necessity.} Let $\widehat{\mathbf{\Theta}}=[\widehat{\Theta}(1),\widehat{\Theta}(2),...,\widehat{\Theta}(L)]\in\mathcal{H}_{\alpha}$ be the closed-loop optimal strategy of Problem (M-SLQ) and $\Phi(\cdot)$ satisfies the following SDE:
\begin{equation*}
    \left\{
    \begin{aligned}
     d\Phi(t)=&\big[A(\alpha_{t})+B(\alpha_{t})\widehat{\Theta}(\alpha_{t})\big]\Phi(t)dt
     +\big[C(\alpha_{t})+D(\alpha_{t})\widehat{\Theta}(\alpha_{t})\big]\Phi(t)dW(t)
     \\
      &+\sum_{j=1}^{L}\big[E_{j}(\alpha_{t-})+F_{j}(\alpha_{t-})\widehat{\Theta}(\alpha_{t-})\big]\Phi(t-)d\widetilde{N}_{j}(t),  \quad t\geq 0,\\
     \Phi(0)=&I,\quad \alpha_{0}=i.      
    \end{aligned}
    \right.
\end{equation*}
Then we have
$$
 V(x,i)=J(x,i;u(\cdot;x,i,\widehat{\mathbf{\Theta}}))=\mathbb{E}\int_{0}^{\infty}\big<\widehat{Q}(\alpha_{t})X(t),X(t)\big>dt=\big<P(i)x,x\big>,\quad \forall(x,i)\in\mathbb{R}^{n}\times \mathcal{S},
$$
with
$$
\widehat{Q}(i)=Q(i)+S(i)^{\top}\widehat{\Theta}(i)+\widehat{\Theta}(i)^{\top}S(i)+\widehat{\Theta}(i)^{\top}R(i)\widehat{\Theta}(i),
$$
and
\begin{equation}\label{eq-P}
P(i)=\mathbb{E}\Big[\int_{0}^{\infty}\Phi(t)^{\top}\widehat{Q}(\alpha_{t})\Phi(t)dt\big|\alpha(0)=i\Big],\qquad i\in\mathcal{S}.
\end{equation}
In the following, we shall prove that the $\mathbf{P}=\left[P(1),P(2),\cdots,P(L)\right]$ defined above is a static stabilizing solution of CAREs \eqref{CAREs-SLQ}.

Applying the dynamic programming principle to Problem (M-SLQ)$_{\infty}$, one has %and using It\^o's formula to $\left<P(\alpha_{t})X(t),X(t)\right>$, one has
\begin{align*}
\big<P(i)x,x\big>
&\leq \mathbb{E}\Big\{
    \int_{0}^{t}
    \left<
    \left(
    \begin{matrix}
    Q(\alpha_{s})&S(\alpha_{s})^{\top}\\
    S(\alpha_{s})&R(\alpha_{s})
    \end{matrix}
    \right)
    \left(
    \begin{matrix}
    X(s)\\
    u(s)
    \end{matrix}
    \right),
    \left(
    \begin{matrix}
    X(s)\\
    u(s)
    \end{matrix}
    \right)
    \right>ds+\big<P(\alpha_{t})X(t),X(t)\big>\Big\}\\
&= \mathbb{E}\int_{0}^{t}
    \left<
    \left(
    \begin{matrix}
    \mathcal{M}(P,\alpha_{s})&\mathcal{L}(P,\alpha_{s})\\
    \mathcal{L}(P,\alpha_{s})^{\top}&\mathcal{N}(P,\alpha_{s})
    \end{matrix}
     \right) 
    \left( 
    \begin{matrix}
    X(s)\\
    u(s)
    \end{matrix}
     \right) ,
     \left( 
    \begin{matrix}
    X(s)\\
    u(s)
    \end{matrix}
    \right)
    \right>ds +\big<P(i)x,x\big>, %\quad \forall (x,i)\in\mathbb{R}^{n}\times\mathcal{S} ,\quad \forall u\in\mathcal{U}_{LQ}^{0}(x,i),
\end{align*}
which implies that
$$
\mathbb{E}\int_{0}^{t}
    \left<
    \left(
    \begin{matrix}
    \mathcal{M}(P,\alpha_{s})&\mathcal{L}(P,\alpha_{s})\\
    \mathcal{L}(P,\alpha_{s})^{\top}&\mathcal{N}(P,\alpha_{s})
    \end{matrix}
     \right) 
    \left( 
    \begin{matrix}
    X(s)\\
    u(s)
    \end{matrix}
     \right) ,
     \left( 
    \begin{matrix}
    X(s)\\
    u(s)
    \end{matrix}
    \right)
    \right>ds\geq 0.
$$    
Dividing both sides by $t$ and letting $t\downarrow 0$, we  can further obtain 
$$
\left[
\begin{array}{cc}
  \mathcal{M}(P,i)  &  \mathcal{L}(P,i)  \\
  \mathcal{L}(P,i)^{\top}  & \mathcal{N}(P,i)
\end{array}
\right]\geq 0,\quad\forall i\in\mathcal{S}.
$$
Consequently, by Lemma \ref{lem-Schur}, the above equation further implies
\begin{equation}\label{SLQ_0-closed-p2}
 \mathcal{M}(P,i)-\mathcal{L}(P,i) \mathcal{N}(P,i)^{\dag} \mathcal{L}(P,i)^{\top} \geq 0, \quad\forall i\in\mathcal{S},
\end{equation}
and
\begin{equation}\label{SLQ_0-closed-p3}
  \mathcal{N}(P,i)\geq 0, \quad \mathcal{L}(P,i)\left[I-\mathcal{N}(P,i)\mathcal{N}(P,i)^{\dag}\right]=0,\quad\forall i\in\mathcal{S}.
\end{equation}
% Next, we will prove that $\mathbf{P}$ defined in \eqref{eq-P} is the solution of \eqref{CAREs-SLQ}. 

On the other hand, by Remark \ref{rmk-useful} and equation \eqref{SLQ_0-closed-p3}, we obtain
\begin{align*}
&\quad  V(x,i)=J(x,i;u(\cdot;x,i,\widehat{\mathbf{\Theta}}))  \\
&=\big<P(i)x,x\big>+\mathbb{E}
    \int_{0}^{\infty}
    \Big<
    \big[\mathcal{M}(P,\alpha_{t})+\mathcal{L}(P,\alpha_{t})\widehat{\Theta}(\alpha_{t})
    +\widehat{\Theta}(\alpha_{t})^{\top}\mathcal{L}(P,\alpha_{t})^{\top}\\
    &\qquad+\widehat{\Theta}(\alpha_{t})^{\top}\mathcal{N}(P,\alpha_{t})\widehat{\Theta}(\alpha_{t})\big]X(t), X(t)\Big>dt\\
&=\big<P(i)x,x\big>+\mathbb{E} \int_{0}^{\infty}
\Big<\Big[\mathcal{M}(P,\alpha_{t})-\mathcal{L}(P,\alpha_{t})
\mathcal{N}(P,\alpha_{t})^{\dag}\mathcal{L}(P,\alpha_{t})^{\top}\Big]X(t),X(t)\Big>dt \\
&\quad+\mathbb{E} \int_{0}^{\infty}\big<\mathcal{N}(P,\alpha_{t})\big[\widehat{\Theta}(\alpha_{t})
+\mathcal{N}(P,\alpha_{t})^{\dag}\mathcal{L}(P,\alpha_{t})^{\top}\big]X(t),\big[\widehat{\Theta}(\alpha_{t})
+\mathcal{N}(P,\alpha_{t})^{\dag}\mathcal{L}(P,\alpha_{t})^{\top}\big]X(t)\big>dt.
\end{align*}
Combining with \eqref{SLQ_0-closed-p2}-\eqref{SLQ_0-closed-p3}, one has
$$\mathcal{M}(P,\alpha_{t})-\mathcal{L}(P,\alpha_{t}) \mathcal{N}(P,\alpha_{t})^{\dag} \mathcal{L}(P,\alpha_{t})^{\top}=0, \quad a.e.\quad a.s..$$
Hence, the  $\mathbf{P}=\left[P(1),P(2),...,P(L)\right]$ defined in \eqref{eq-P} solves the CAREs \eqref{CAREs-SLQ},
and satisfies
$$\mathcal{N}(P,\alpha_{t})^{\frac{1}{2}}\big[\widehat{\Theta}(\alpha_{t})
+\mathcal{N}(P,\alpha_{t})^{\dag}\mathcal{L}(P,\alpha_{t})^{\top}\big]=0,\quad a.e.\quad a.s.,$$
which, together with \eqref{SLQ_0-closed-p3}, derives 
$$\mathcal{N}(P,i)\widehat{\Theta}(i)+\mathcal{L}(P,i)^{\top}=0, \quad\forall i\in\mathcal{S}.$$
Since $\mathcal{N}(P,i)\mathcal{N}(P,i)^{\dag}$ is an orthogonal projection, we have
$$\widehat{\Theta}(i)=-\mathcal{N}(P,i)^{\dag}\mathcal{L}(P,i)^{\top}+\big[I-\mathcal{N}(P,i)^{\dag}\mathcal{N}(P,i)\big]\Pi(i),\quad \forall i\in\mathcal{S},$$
for some $\mathbf{\Pi}=\left[\Pi(1),\Pi(2),...,\Pi(L)\right]\in\mathcal{D}\left(\mathbb{R}^{m\times n}\right)$, which further implies that $\mathbf{P}=\left[P(1),P(2),...,P(L)\right]$ is a  static stabilizing solution of CAREs \eqref{CAREs-SLQ}.

\textbf{Sufficiency.}  For any $u(\cdot)\in\mathcal{U}_{ad}(x,i)$, let $X(\cdot;x,i,u(\cdot))$ be the corresponding controlled state process, $\widehat{\mathbf{\Theta}}$ be given by \eqref{eq-Theta} and $\mathbf{P}\in\mathcal{D}\left(\mathbb{S}^{n}\right)$ be a static stabilizing solution to \eqref{CAREs-SLQ}. Additionally, we define
$$
\nu(t)=u(t)-\widehat{\Theta}(\alpha_{t-})X(t-;x,i,u(\cdot)),\qquad t\geq 0.
$$
Then by the uniqueness solvability of SDE \eqref{intro-state}, the solution $X(\cdot;x,i,u(\cdot))$ also solves the following SDE:
\begin{equation}\label{state-SLQ-closed-3}
 \left\{
 \begin{aligned}
   dX(t)&=\left\{\left[A(\alpha_{t})+B(\alpha_{t})\widehat{\Theta}(\alpha_{t})\right]X(t)+B(\alpha_{t})\nu(t)\right\}dt\\
   &\quad +\left\{\left[C(\alpha_{t})+D(\alpha_{t})\widehat{\Theta}(\alpha_{t})\right]X(t)+D(\alpha_{t})\nu(t)\right\}dW(t)\\
   &\quad+\sum_{j=1}^{L}\left\{\left[E_{j}(\alpha_{t-})+F_{j}(\alpha_{t-})\widehat{\Theta}(\alpha_{t-})\right]X(t-)+F_{j}(\alpha_{t-})\nu(t)\right\}d\widetilde{N}_{j}(t),\qquad t\geq0,\\
   X(0)&=x,\quad \alpha(0)=i.
   \end{aligned}
 \right.
 \end{equation}
Consequently, we have
\begin{equation}\label{eq-cost}
\begin{aligned}
J(x,i;u(\cdot))&=\mathbb{E}\int_{0}^{\infty}
    \left<
    \left(
    \begin{matrix}
    Q(\alpha_{t})&S(\alpha_{t})^{\top}\\
    S(\alpha_{t})&R(\alpha_{t})
    \end{matrix}
     \right) 
    \left( 
    \begin{matrix}
    X(t)\\
    \widehat{\Theta}(\alpha_{t})X(t)+\nu(t)
    \end{matrix}
     \right) ,
     \left( 
    \begin{matrix}
    X(t)\\
    \widehat{\Theta}(\alpha_{t})X(t)+\nu(t)
    \end{matrix}
    \right)
    \right>dt\\
    &=\mathbb{E}\int_{0}^{\infty}\Big[\big<\widehat{Q}(\alpha_{t})X(t),X(t)\big>+\big<R(\alpha_{t})\nu(t),\nu(t)\big>
    +2\big<\big(S(\alpha_{t})+R(\alpha_{t})\widehat{\Theta}(\alpha_{t})\big)X(t),\nu(t)\big>\Big]dt.
\end{aligned}
\end{equation}

On the other hand, applying It\^o's rule to $\big<P(\alpha_{t})X(t),X(t)\big>$, we have
\begin{align*}
0&=\big<P(i)x,x\big>+\mathbb{E}\int_{0}^{\infty}\Big[\big<\big(\mathcal{M}(P,\alpha_{t})-Q(\alpha_{t})\big)X(t),X(t)\big>
+\big<\big(\mathcal{N}(P,\alpha_{t})-R(\alpha_{t})\big)\nu(t),\nu(t)\big>\\
   &\quad +2\big<\big(\mathcal{L}(P,\alpha_{t})-S(\alpha_{t})^{\top}\big)\widehat{\Theta}(\alpha_{t})X(t),X(t)\big>
   +2\big<\big(\mathcal{L}(P,\alpha_{t})-S(\alpha_{t})^{\top}\big)\nu(t),X(t)\big>\\
   &\quad+2\big<\big(\mathcal{N}(P,\alpha_{t})-R(\alpha_{t})\big)\widehat{\Theta}(\alpha_{t})X(t),\nu(t)\big>
   +\big<\big(\mathcal{N}(P,\alpha_{t})-R(\alpha_{t})\big)\widehat{\Theta}(\alpha_{t})X(t),\widehat{\Theta}(\alpha_{t})X(t)\big>\Big]dt.
\end{align*}
Substituting the above equation into \eqref{eq-cost} yields
\begin{equation}\label{eq-cost-2}
\begin{aligned}
&\quad J(x,i;u(\cdot))\\
&=\mathbb{E}\int_{0}^{\infty}\Big[\big<\big(\mathcal{M}(P,\alpha_{t})+\mathcal{L}(P,\alpha_{t})\widehat{\Theta}(\alpha_{t})
+\widehat{\Theta}(\alpha_{t})^{\top}\mathcal{L}(P,\alpha_{t})^{\top}+\widehat{\Theta}(\alpha_{t})^{\top}\mathcal{N}(P,\alpha_{t})\widehat{\Theta}(\alpha_{t})\big)X(t),X(t)\big>\\
&\quad+2\big<\big(\mathcal{N}(P,\alpha_{t})\widehat{\Theta}(\alpha_{t})+\mathcal{L}(P,\alpha_{t})^{\top}\big)X(t),\nu(t)\big>
    +\big<\mathcal{N}(P,\alpha_{t})\nu(t),\nu(t)\big>\Big]dt+\big<P(i)x,x\big>.
\end{aligned}
\end{equation}

Note that for any $i\in\mathcal{S}$,
\begin{align*}
  &\mathcal{M}(P,i)+\mathcal{L}(P,i)\widehat{\Theta}(i)+\widehat{\Theta}(i)^{\top}\mathcal{L}(P,i)^{\top}
  +\widehat{\Theta}(i)^{\top}\mathcal{N}(P,i)\widehat{\Theta}(i)=0,\\
  &\mathcal{N}(P,i)\widehat{\Theta}(i)+\mathcal{L}(P,i)^{\top}=0.
\end{align*}
Hence, the equation \eqref{eq-cost-2} implies that 
\begin{equation}\label{eq-cost-3}
\begin{aligned}
 J(x,i;u(\cdot))=\big<P(i)x,x\big>+
\mathbb{E}\int_{0}^{\infty}\big<\mathcal{N}(P,\alpha_{t})\nu(t),\nu(t)\big>dt,\quad \forall (x,i)\in\mathbb{R}^{n}\times\mathcal{S}.
\end{aligned}
\end{equation}

On the other hand, let $u(\cdot;x,i,\widehat{\mathbf{\Theta}})$ be the outcome of closed-loop optimal strategy $\widehat{\mathbf{\Theta}}$. Then one can similarly derive that
\begin{equation}\label{eq-cost-4}
\begin{aligned}
 J(x,i;u(\cdot;x,i,\widehat{\mathbf{\Theta}}))=\big<P(i)x,x\big>,\quad \forall (x,i)\in\mathbb{R}^{n}\times\mathcal{S}.
\end{aligned}
\end{equation}
Combining with the fact $\mathcal{N}(P,i)\geq 0$ for every $i\in\mathcal{S}$, we can derive that the $\widehat{\mathbf{\Theta}}$ defined in \eqref{eq-Theta} is a closed-loop optimal strategy by definition. Consequently,  the value function is given by 
$$
V(x,i)=J(x,i;u(\cdot;x,i,\widehat{\mathbf{\Theta}}))=\big<P(i)x,x\big>.
$$ 
This completes the proof.
\end{proof}

\begin{remark}\rm
 Since the value function of Problem (M-SLQ) is unique. Thus, from \eqref{eq-value-function}, one can obtain that the CAREs \eqref{CAREs-SLQ} admits at most one static stabilizing solution. 
\end{remark}

% ======================================================================
%  Section 5: Applications -- tracking a wealth level in a
%  regime-switching market with price jumps at the switching times.
%  Notation follows the main text of M-SLQ-infty.tex.
%  Citation keys refer to refs-section5.bib.
% ======================================================================
\section{Application to lifetime wealth tracking problem}\label{sec-4}
Based on the derived results, we shall investigate a class of lifetime wealth tracking problems in this section, which has already attracted extensive research attention (see, for example, Yao et al. \cite{YaoZhangZhou2006} and Hu et al. \cite{HuShiXu2022cocv}). In contrast to the existing literature, we assume that the prices of risky asset may jump at the moments of market regime switches. This jump magnitudes can be interpreted as the additional gains or losses of asset value incurred during the regime-switching process. Indeed, in the field of mathematical finance, a substantial body of literature has documented that asset prices exhibit jump-like fluctuations in response to changes in market regimes (see, for example, Shi and Xu \cite{shi2025optimal},  and Zhang et al. \cite{Zhang-Siu-Meng-2010}).

Specifically, we use a Markov chain $\alpha(\cdot)$ to characterize the market regimes, and assume that there are a risk-free asset and a risky asset in the market. The price of the risk-free asset satisfies the following equation
\[
dS_0(t)=r(\alpha_t)S_0(t)\,dt,\qquad S_0(0)=s_0>0,
\]
where $r(i)$ represents the interest rate while the market regime is in $i$. 
The risky asset is governed by
\[
\left\{
\begin{aligned}
dS(t)&=S(t-)\Big[\mu(\alpha_t)\,dt+\sigma(\alpha_t)\,dW(t)
+\sum_{j=1}^{L}\gamma_j(\alpha_{t-})\,d\widetilde{N}_j(t)\Big],\\
S(0)&=s>0,
\end{aligned}
\right.
\]
where $\mu(i)$ is the appreciation rate, $\sigma(i)$ is the volatility, and
$\gamma_j(i)>-1$ is the relative price jump when the regime switches from
$i$ to $j$. Therefore, when the market regime switches from another state to state $j$ at time $t$, the price of the risky asset undergoes a jump of magnitude $S(t-)\gamma_{j}(\alpha_{t-})$. This jump magnitude depends not only on the price of the risky asset prior to the regime switch, but also on the market states before and after the switch. For example, suppose the market has only two regimes, regime $1$ represents a bull market and regime $2$ a bear market. Then $\gamma_2(1)<0$ models a crash at the
onset of the bear regime, while $\gamma_1(2)\geq0$ allows for a rebound at
recovery. If $\gamma_j(i)\equiv0$, the model reduces to the standard
regime-switching diffusion market without price jumps at switching times.

%\subsection{The wealth process and the tracking problem}

Let $u(t)$ be the dollar amount invested in the risky asset at time $t$, with the remaining wealth invested in the risk-free asset. Then the corresponding self-financing wealth process satisfies
\begin{equation}\label{app-wealth}
\left\{
\begin{aligned}
dX(t)&=\big[r(\alpha_{t})X(t)+\big(\mu(\alpha_{t})-r(\alpha_{t})\big)u(t)\big]dt
+\sigma(\alpha_{t})u(t)dW(t)
+\sum_{j=1}^{L}\gamma_{j}(\alpha_{t-})u(t)d\widetilde{N}_{j}(t),\\
X(0)&=x,\quad \alpha_{0}=i.
\end{aligned}
\right.
\end{equation}
Following the benchmark-tracking formulation introduced in
\cite{YaoZhangZhou2006,HuShiXu2022cocv}, we assume that the investor aims to keep the wealth close to the target
$$
G(t)\triangleq d\,e^{\int_{0}^{t}r(\alpha_{s})ds},\qquad d\in\mathbb{R},
$$
which can be viewed as the value at time $t$ of the initial amount $d$ invested in the risk-free asset. Then the lifetime wealth tracking problem can be summarized as

\textbf{Problem (LWT).} For any given $(x,d,i)\in \mathbb{R}\times\mathbb{R}\times\mathcal{S}$, find a optimal investment control $u^{*}(\cdot)$ to minimize the following performance functional
\begin{equation}\label{app-objective}
 J(x,d,i;u(\cdot))\triangleq
\mathbb{E}\int_{0}^{\infty}e^{-2\int_{0}^{t}\rho(\alpha_{s})ds}
\Big[\big(X(t)-G(t)\big)^{2}+\lambda u(t)^{2}\Big]dt,
\end{equation}
where $\lambda>0$ is the penalty coefficient for the risky position, and $\rho(i)$ denotes the discount rate under regime $i$.
%We assume that
%\begin{equation}\label{app-wellposed}
%\rho(i)-r(i)\geq c>0,\qquad i\in\mathcal{S},
%\end{equation}
%for some constant $c$, which ensures well-posedness on the infinite horizon.
%\subsection{Reduction to Problem (M-SLQ)$_{\infty}$}

Clearly, the problem \eqref{app-objective}  is a discounted optimization problem. Using the linear substitution method introduced in Wu et al. \cite{WuLiZhang2025}, we can further convert it into a class of undiscounted optimization problems. To this end, we define
\begin{equation}\label{app-transform}
\widetilde{X}(t)\triangleq e^{-\int_{0}^{t}\rho(\alpha_{s})ds}\big(X(t)-G(t)\big),
\qquad
\widetilde{u}(t)\triangleq e^{-\int_{0}^{t}\rho(\alpha_{s})ds}u(t).
\end{equation}
%Since the discount factor and $G(\cdot)$ are continuous processes of finite
%variation, $Y(\cdot)$ jumps
%only through $X(\cdot)$: at a switching time from $i$ to $j$,
%\(
%\Delta Y(t)=\gamma_j(i)\widetilde u(t).
%\)
Then applying It\^o's formula, we obtain
\begin{equation}\label{app-Y-raw}
\left\{
\begin{aligned}
d\widetilde{X}(t)&=\Big[\big(r(\alpha_{t})-\rho(\alpha_{t})\big)\widetilde{X}(t)
+\big(\mu(\alpha_{t})-r(\alpha_{t})\big)\widetilde{u}(t)\Big]dt
+\sigma(\alpha_{t})\widetilde{u}(t)dW(t)
+\sum_{j=1}^{L}\gamma_{j}(\alpha_{t-})\widetilde{u}(t)d\widetilde{N}_{j}(t),\\
\widetilde{X}(0)&=x-d.
\end{aligned}
\right.
\end{equation}
Additionally, the performance functional \eqref{app-objective} can be rewritten as 
\begin{equation}\label{app-cost}
J(x,d,i;u(\cdot))=\widetilde{J}(x-d,i;\widetilde{u}(\cdot))\triangleq\mathbb{E}\int_{0}^{\infty}\big[\widetilde{X}(t)^{2}+\lambda \widetilde{u}(t)^{2}\big]dt .
\end{equation}
Consequently, the Problem (LWT)  is transformed into a class of infinite-horizon SLQ control problems for Markov regime-switching jump-diffusion systems with the following coefficients:
\begin{equation}\label{AP-coefficients}
\left\{
\begin{aligned}
 & A(i)=r(i)-\rho(i),\quad C(i)=E_{j}(i)=0,\quad B(i)=\mu(i)-r(i),\quad D(i)=\sigma(i),\quad F_{j}(i)=\gamma_{j}(i),\\
 &Q(i)=1,\quad S(i)=0,\quad R(i)=\lambda,\qquad \forall\, i,j\in\mathcal{S}.
\end{aligned}
\right.
\end{equation}
In the following, we denote the SLQ control problem with state equation \eqref{app-Y-raw} and cost functional \eqref{app-cost} as Problem (M-SLQ)$_{\infty}^{'}$, and introduce the following assumption in this section:\\
\textbf{(H1)$'$} $\rho(i)>r(i)$ for any $i\in\mathcal{S}$.\\
Then, by Remark \ref{rmk-stable}, one can easily verify that the system \eqref{app-Y-raw} is $L^{2}$-stable based on assumption (H1)$'$, and  the closed-loop strategy set for Problem (M-SLQ)$_{\infty}^{'}$ is non-empty. Additionally, the operators introduced in \eqref{MLN} are given by
\begin{equation}\label{app-MLN}
\begin{aligned}
\mathcal{M}(\mathbf P,i)&=2\big(r(i)-\rho(i)\big)P(i)+1
+\sum_{j\neq i}\pi_{ij}\big(P(j)-P(i)\big),\\
\mathcal{L}(\mathbf P,i)&=\big(\mu(i)-r(i)\big)P(i)
+\sum_{j\neq i}\pi_{ij}\gamma_{j}(i)\big(P(j)-P(i)\big),\\
\mathcal{N}(\mathbf P,i)&=\lambda+\sigma(i)^{2}P(i)
+\sum_{j\neq i}\pi_{ij}\gamma_{j}(i)^{2}P(j).
\end{aligned}
\end{equation}
 Consequently, by Theorem \ref{thm-SLQ_0-closed}, we have the following result directly. 

\begin{proposition}
If the following CAREs 
\begin{equation}\label{app-CARE}
\left\{
\begin{aligned}
\mathcal{M}(\mathbf P,i)-\frac{\mathcal{L}(\mathbf P,i)^{2}}{\mathcal{N}(\mathbf P,i)}=0,\\
\mathcal{N}(\mathbf P,i)>0,\qquad i\in\mathcal{S},
\end{aligned}
\right.
\end{equation}
admits a solution $\mathbf{P}\in\mathcal{D}(\mathbb{S}^{n})$ such that $\widehat{\mathbf{\Theta}}=\left(\widehat{\Theta}(1),\widehat{\Theta}(2),\cdots,\widehat{\Theta}(L)\right)
\in\mathcal{H}_{\alpha}$ with
\begin{equation}\label{app-theta}
\widehat{\Theta}(i)=-\frac{\mathcal{L}(\mathbf P,i)}{\mathcal{N}(\mathbf P,i)},\quad i\in\mathcal{S}.
\end{equation}
Then the Problem (LWT) admits an optimal investment  control
\begin{equation}\label{app-optimal-control}
u^{*}(t)=\widehat{\Theta}(\alpha_{t-})(X^{*}(t-)-G(t-)),\qquad t\geq 0,
\end{equation}
where $X^{*}(\cdot)$ is the solution of the following SDE:
\begin{equation}\label{app-optimal-wealth}
  \left\{
\begin{aligned}
dX^{*}(t)&=\big[r(\alpha_{t})X^{*}(t)+\big(\mu(\alpha_{t})-r(\alpha_{t})\big)\widehat{\Theta}(\alpha_{t})(X^{*}(t)-G(t))\big]dt
+\sigma(\alpha_{t})\widehat{\Theta}(\alpha_{t})(X^{*}(t)-G(t))dW(t)\\
&\quad+\sum_{j=1}^{L}\gamma_{j}(\alpha_{t-})\widehat{\Theta}(\alpha_{t-})(X^{*}(t-)-G(t-))d\widetilde{N}_{j}(t),\\
X^{*}(0)&=x,\quad \alpha_{0}=i.
\end{aligned}
\right.
\end{equation}
In this case, the optimal value function is given by
\begin{equation}\label{app-value}
J(x,d,i;u^{*}(\cdot))=P(i)(x-d)^{2}.
\end{equation}
\end{proposition}

\begin{proof}
Obviously, if the CAREs \eqref{app-CARE} admits a solution $\mathbf{P}$ such that the corresponding $\widehat{\mathbf{\Theta}}\in\mathcal{H}_{\alpha}$, then it must be the unique static stabilizing solution. Consequently, by Theorem \ref{thm-SLQ_0-closed}, the Problem (M-SLQ)$_{\infty}^{'}$ admits the optimal control
$$
\widetilde{u}(t)=\widehat{\Theta}(\alpha_{t-})\widetilde{X}^{*}(t-),\quad t\geq 0,
$$
with $\widetilde{X}^{*}(\cdot)$ being the solution of SDE:
\begin{equation}\label{app-Y-raw-star}
\left\{
\begin{aligned}
d\widetilde{X}^{*}(t)&=\Big[\big(r(\alpha_{t})-\rho(\alpha_{t})\big)\widetilde{X}^{*}(t)
+\big(\mu(\alpha_{t})-r(\alpha_{t})\big)\widehat{\Theta}(\alpha_{t})\widetilde{X}^{*}(t)\Big]dt
+\sigma(\alpha_{t})\widehat{\Theta}(\alpha_{t})\widetilde{X}^{*}(t)dW(t)\\
&\quad+\sum_{j=1}^{L}\gamma_{j}(\alpha_{t-})\widehat{\Theta}(\alpha_{t-})\widetilde{X}^{*}(t-)d\widetilde{N}_{j}(t),\\
\widetilde{X}^{*}(0)&=x-d.
\end{aligned}
\right.
\end{equation}
In this case, the value function of Problem (M-SLQ)$_{\infty}^{'}$ is given by 
$$\widetilde{J}(x-d,i;\widetilde{u}^{*}(\cdot))=P(i)(x-d)^{2}.$$
Then the desired results follows from the relations \eqref{app-transform} and \eqref{app-cost}. This completes the proof.
\end{proof}

\begin{remark}\rm
If $\gamma_j(i)\equiv0$ and $\lambda=0$, then \eqref{app-CARE} becomes
$$
2\big(r(i)-\rho(i)\big)P(i)+1+\sum_{j\neq i}\pi_{ij}\big(P(j)-P(i)\big)
-\frac{\big(\mu(i)-r(i)\big)^2P(i)^2}{\sigma(i)^2P(i)}=0,
\qquad i\in\mathcal S.
$$
We point out that the above linear system has been solved in Remark 6.5 of Hu et al. \cite{HuShiXu2022cocv} for the unconstrained tracking problem without price jumps, which can be considered as a special case of Problem (LWT). %Thus the additional terms in \eqref{app-MLN} are exactly those involving $\gamma_{j}(i)$.
\end{remark}

\section{Numerical analysis}\label{sec-5}
In this section, we consider a two-regime market, where regime $1$ represents a bull market
and regime $2$ represents a bear market.  The generator of the Markov chain is given by
$$
\mathbf{\Pi}=\begin{pmatrix}
               -0.2 & 0.2 \\
               0.1 & -0.1 
             \end{pmatrix},
$$
and the model parameters are given in the following table:
\begin{table}[H]
\centering
%\caption{Market parameters in the two-regime example.}
\label{tab-app-parameters}
\begin{tabular}{@{}c c c c c c c @{}} 
\toprule
Regime $i$ & $r(i)$ & $\mu(i)$ & $\sigma(i)$ & $\rho(i)$
& $\gamma_1(i)$ & $\gamma_2(i)$  \\
\midrule
$1$ & $0.03$ & $0.12$ & $0.12$ & $0.05$ & $0$ & $-0.18$ \\
$2$ & $0.01$ & $-0.04$ & $0.20$ & $0.03$ & $0.05$ & $0$ \\
\bottomrule
\end{tabular}
\end{table}

\noindent Additionally, we take penalty weight $\lambda=0.02$. % A switch from the bull regime to the bear regime is
%accompanied by an $18\%$ price drop, while a switch back to the bull regime produces a $5\%$ rebound. 
%The jump-adjusted excess returns are
%\(
%\widetilde b(1)=0.08-0.03-0.18=-0.13,
%\widetilde b(2)=0.06-0.01+2\times0.05=0.15.
%\)
%Hence crash risk makes the adjusted excess return negative in the bull
%regime.
Based on the numerical values provided above, CAREs \eqref{app-CARE} is transformed into:
\begin{equation}\label{exam-CAREs}
\left\{
\begin{aligned}
&-0.24P(1)+0.2P(2)+1-\frac{\left[0.09P(1)-0.2\times 0.18\times(P(2)-P(1))\right]^2}{0.02+0.12^{2}\times P(1)+0.2\times 0.18^2\times P(2)}=0,\\
&0.1P(1)-0.14P(2)+1-\frac{\left[-0.05P(2)+0.1\times 0.05\times(P(1)-P(2))\right]^2}{0.02+0.2^{2}\times P(2)+0.1\times 0.05^2\times P(1)}=0,\\
&0.02+0.12^{2}\times P(1)+0.2\times 0.18^2\times P(2)>0,\\
&0.02+0.2^{2}\times P(2)+0.1\times 0.05^2\times P(1)>0.
\end{aligned}
\right.
\end{equation}
%Based on the above numerical example, we can obtain the numerical solution of CAREs \eqref{app-CARE}
By applying Newton's iterative algorithm, the numerical solution to the above equation is obtained as follows:
\[
P(1)=5.5867254,\qquad P(2)=7.72632362,
\]
We point out that the residual of above numerical solution is below $10^{-9}$.
% Clearly, the above numerical solution satisfies
%\[
%\mathcal N(\mathbf P,1)={0.1505}>0,\qquad   
%\mathcal N(\mathbf P,2)={0.3304}>0.
%\]
Additionally, substituting the numerical solution into equation \eqref{app-theta} yields
\[
\widehat\Theta(1)={-2.8288},\qquad     
\widehat\Theta(2)={1.2014}.
\]
One can easily verify that
\begin{equation}\label{app-stable}
\left\{
\begin{aligned}
&2[r(1)-\rho(1)+(\mu(1)-r(1))\widehat\Theta(1)]+\sigma(1)^{2}\widehat\Theta(1)^{2}
+\pi_{12}\gamma_{2}(1)^{2}\widehat\Theta(1)^{2}=-0.3821<0,\\
&2[r(2)-\rho(2)+(\mu(2)-r(2))\widehat\Theta(2)]+\sigma(2)^{2}\widehat\Theta(2)^{2}
+\pi_{21}\gamma_{1}(2)^{2}\widehat\Theta(2)^{2}=-0.1020<0.
\end{aligned}
\right.
\end{equation}
Hence, by Remark \ref{rmk-stable}, we can claim that $\widehat{\mathbf{\Theta}}=(\widehat\Theta(1),\widehat\Theta(2))\in\mathcal{H}_{\alpha}.$
%To check stabilizability, we apply Proposition \ref{prop-L2} to the
%closed-loop system. Let 
%To verify the stabilizing property, define the closed-loop coefficients
%\[
%A^{\widehat\Theta}(i)=A(i)+B(i)\widehat\Theta(i),\qquad
%C^{\widehat\Theta}(i)=D(i)\widehat\Theta(i),\qquad
%E_j^{\widehat\Theta}(i)=F_j(i)\widehat\Theta(i).
%\]
%Taking $p(1)=p(2)=1$, the Lyapunov inequalities in Proposition
%\ref{prop-L2} reduce to
%\[
%2A^{\widehat\Theta}(1)+\big(C^{\widehat\Theta}(1)\big)^2
%+\pi_{12}\big(E_2^{\widehat\Theta}(1)\big)^2
%=-0.3905<0,
%\]
%and
%\[
%2A^{\widehat\Theta}(2)+\big(C^{\widehat\Theta}(2)\big)^2
%+\pi_{21}\big(E_1^{\widehat\Theta}(2)\big)^2
%=-0.5326<0.
%\]
%Hence $\widehat{\mathbf{\Theta}}\in\mathcal{H}_{\alpha}$, and $\mathbf P$
%is a static stabilizing solution of \eqref{CAREs-SLQ}. 
Consequently, the
optimal portfolio is given by \eqref{app-optimal-control}.%, with optimal value

\begin{remark}\rm
  To isolate the effect of switching-time price jumps, we also compute a jump-free benchmark by setting $\gamma_j(i)\equiv0$, while keeping the same generator and market parameters. By applying Newton's iterative algorithm again, we obtain
\[
%(P(1),P(2))=(3.6707,\,6.8949),\qquad\text{and}\qquad
(\widehat\Theta(1),\widehat\Theta(2))=(-4.5343,\,1.1655).
\]
Compared with the model that includes switching-time jump, the feedback strategy in the bull regime becomes substantially more aggressive (in particular, $|\widehat\Theta(1)|$ increases from $2.8288$ to $4.5343$). This indicates that accounting for crash risk at regime switches leads to a more conservative optimal strategy in the bull regime. Intuitively, the investor reduces the risky position to limit the potential loss caused by the $18\%$ price drop ($\gamma_2(1)=-0.18$) upon a bull-to-bear regime
switch.
\end{remark}

In the following, we present the dynamic behavior of the optimal state and control, and further analyze the sensitivity of the stabilizing solution to variations in the control weight $\lambda$, with all other parameters held fixed.

\begin{figure}[H]
\centering
\includegraphics[width=0.82\textwidth]{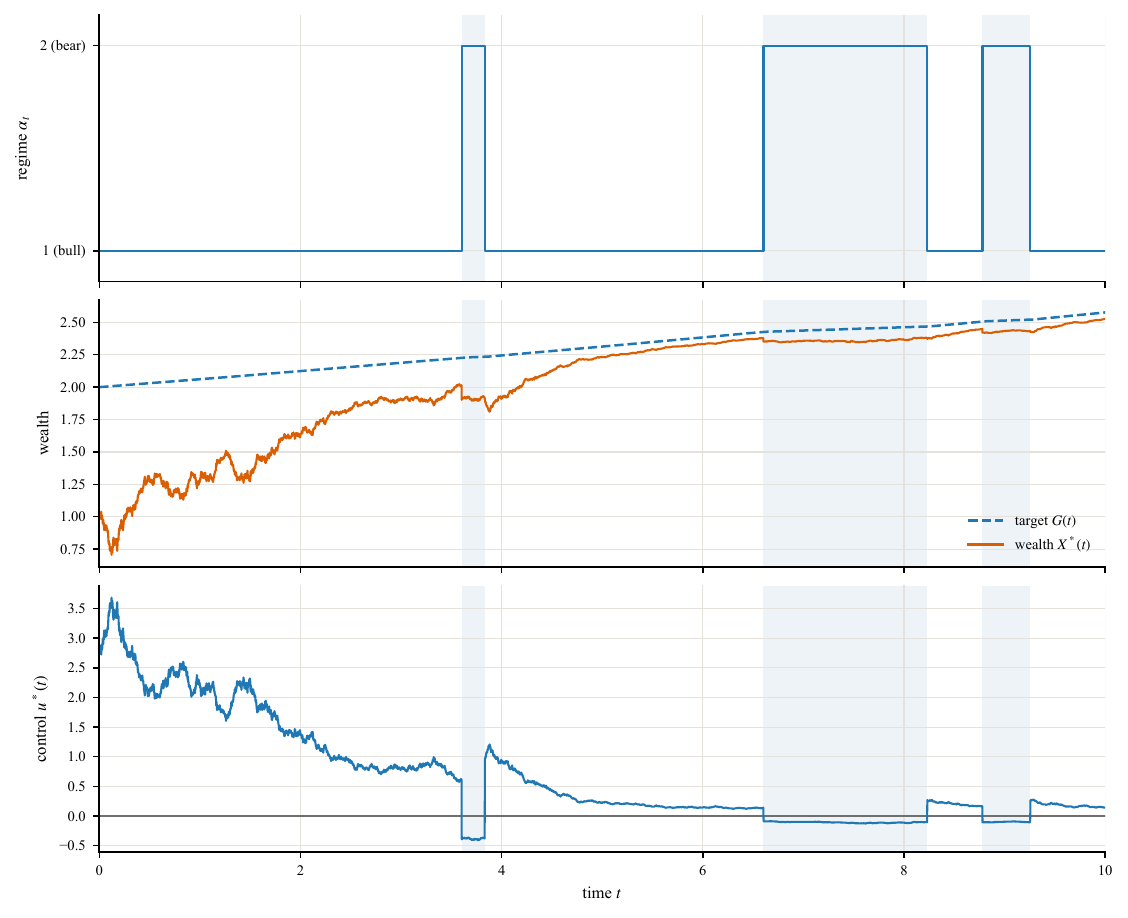}
\caption{Sample paths of the optimal wealth $X^*(t)$, target $G(t)$, and investment control $u^*(t)$.}
\label{fig-tracking}
\end{figure}

%Several features of the sample paths are worth highlighting. Recall from \eqref{app-optimal-control} that the optimal strategy is a linear feedback policy on the tracking error
%$X^{*}(t-)-G(t-)$:
%\[
%u^{*}(t)=\widehat\Theta(\alpha_{t-})(X^{*}(t-)-G(t-)),
%\]
%where the feedback coefficient $\widehat\Theta$ switches with
%the market regime.
%\begin{itemize}
%\item The signs $\widehat\Theta(1)=-2.8288<0$ and $\widehat\Theta(2)=1.2014>0$ imply regime-dependent positions.
%If $X^{*}(t)<G(t)$, the investor is long in the bull regime and short in the bear regime, the latter being consistent with
%the negative excess return $\mu(2)-r(2)<0$. When $X^{*}(t)>G(t)$, the positions reverse, so the strategy is
%mean-reverting toward the benchmark.
%\item Since $u^{*}(t)$ is proportional to $X^{*}(t)-G(t)$, the position size decreases as the tracking gap narrows. 
%This explains why the control path converges to zero once the wealth process approaches the benchmark.
%\end{itemize}
\begin{figure}[H]
\centering
\subfloat[Feedback coefficients.\label{fig-lambda-feedback-gain}]{
\includegraphics[width=0.48\textwidth]{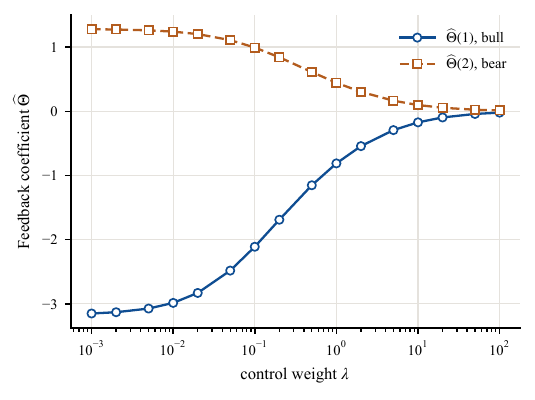}}
\hfill
\subfloat[Riccati solutions.\label{fig-lambda-value-coefficients}]{
\includegraphics[width=0.48\textwidth]{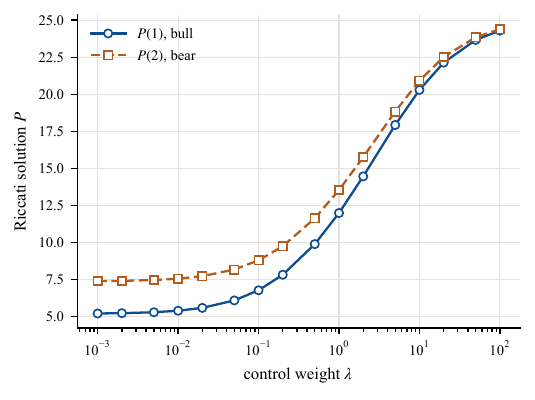}}
\caption{Sensitivity of the stabilizing solution with respect to the control weight $\lambda$.}
\label{fig-lambda-sensitivity}
\end{figure}

Figure~\ref{fig-tracking} illustrates a representative trajectory for the Markov chain process $\alpha(\cdot)$, the optimal wealth process $X^{*}(\cdot)$, and the optimal investment control $u^{*}(\cdot)$. 
Here, we slightly abuse the terminology by referring to $|u^{*}(\cdot)|$ as the position level in the risky asset. From the plots of the optimal strategy and the optimal state dynamics, it is evident that the optimal position $|u^{*}(t)|$ in the risky asset becomes increasingly smaller as the tracking difference $X^{*}(t)-G(t)$ of wealth diminishes. In particular, as the tracking difference $X^{*}(t)-G(t)$ tends to zero, the optimal position $|u^{*}(t)|$ in the risky asset also approaches zero. This is intuitively appealing, since the wealth target $G(t)$ is designed to grow at the risk-free rate $r(\alpha_{t})$; once the investor’s wealth level $X(t)$ reaches the target $G(t)$ at time $t$, the optimal amount $u^{*}(t)$ invested in the risky asset must necessarily drop to zero.

Figure~\ref{fig-lambda-sensitivity} reports the corresponding feedback coefficients and Riccati solutions.
As $\lambda$ increases, both $|\widehat\Theta(1)|$ and $|\widehat\Theta(2)|$ decrease. %For example, $(\widehat\Theta(1),\widehat\Theta(2))=(-2.8288,\,1.2014)$ for $\lambda=0.02$, whereas $(\widehat\Theta(1),\widehat\Theta(2))=(-0.8118,\,0.4433)$ for $\lambda=1$.
 Hence, a larger control penalty leads to a more conservative optimal risky position in both regimes. In contrast, $P(1)$ and $P(2)$ increase with $\lambda$, which indicates a higher tracking cost when control becomes more expensive. Moreover, a larger $\lambda$ drives the optimal risky position toward zero. Consequently, the regime-dependent coefficients $\mu(\alpha_{t}),\sigma(\alpha_{t}),\{\gamma_j(\alpha_{t-})\}_{j\in\mathcal S}$ have a weaker effect on the value function, and the gap between $P(1)$ and $P(2)$ becomes smaller.

\bibliography{references}
\bibliographystyle{plainnat}

\end{document}